\documentclass[preprint,12pt]{elsarticle}

\usepackage[margin=1in]{geometry} 
\usepackage{amsmath,amsthm,amssymb}
\usepackage{graphicx}
\usepackage{color}
\usepackage[dvipsnames]{xcolor}
\usepackage{subcaption}
\usepackage[export]{adjustbox}
\usepackage{algorithm}
\usepackage{algorithmicx}
\usepackage{lmodern}
\usepackage[noend]{algpseudocode}
\usepackage{tikz}
\usepackage[toc,page]{appendix}
\usetikzlibrary{bayesnet}

\newcommand{\norm}[1]{\left\|#1\right\|}

\algnewcommand\algorithmicinput{\textbf{INPUT: }}
\algnewcommand\Input{\item[\algorithmicinput]}
\algnewcommand\algorithmicoutput{\textbf{OUTPUT: }}
\algnewcommand\Output{\item[\algorithmicoutput]}

\newcommand\sForAll[2]{ \ForAll{#1}#2\EndFor} 

\newtheorem{assumption}{Assumption}
\newtheorem{lemma}{Lemma}
\newtheorem{corollary}{Corollary}
\newtheorem{theorem}{Theorem}

\graphicspath{ {./pic/} }

\begin{document}
	

\begin{frontmatter}
\title{Bayesian Sparse learning with preconditioned stochastic gradient MCMC and its applications }
\author[a]{Yating Wang}
\author[a]{Wei Deng}
\author[b]{Guang Lin\corref{cor1}}
\cortext[cor1]{Corresponding author}
\ead {guanglin@purdue.edu}
\address[a]{Department of Mathematics, Purdue University, West Lafayette, IN 47907, USA}
\address[b]{Department of Mathematics, School of Mechanical Engineering, Department of Statistics (Courtesy), Department of Earth, Atmospheric, and Planetary Sciences (Courtesy), Purdue University, West Lafayette, IN 47907, USA}
\begin{abstract}
Deep neural networks have been successfully employed in an extensive variety of research areas, including solving partial differential equations. Despite its significant success, there are some challenges in effectively training DNN, such as avoiding over-fitting in over-parameterized DNNs and accelerating the optimization in DNNs with pathological curvature. 
 In this work, we propose a Bayesian type sparse deep learning algorithm. The algorithm utilizes a set of spike-and-slab priors for the parameters in the deep neural network. The hierarchical Bayesian mixture will be trained using an adaptive empirical method. That is, one will alternatively sample from the posterior using preconditioned stochastic gradient Langevin Dynamics (PSGLD), and optimize the latent variables via stochastic approximation. The sparsity of the network is achieved while optimizing the hyperparameters with adaptive searching and penalizing. A popular SG-MCMC approach is Stochastic gradient Langevin dynamics (SGLD). However, considering the complex geometry in the model parameter space in non-convex learning, updating parameters using a universal step size in each component as in SGLD may cause slow mixing. To address this issue, we apply a computationally manageable preconditioner in the updating rule, which provides a step-size parameter to adapt to local geometric properties. Moreover, by smoothly optimizing the hyperparameter in the preconditioning matrix, our proposed algorithm ensures a decreasing bias, which is introduced by ignoring the correction term in preconditioned SGLD. According to the existing theoretical framework, we show that the proposed algorithm can asymptotically converge to the correct distribution with a controllable bias under mild conditions. Numerical tests are performed on both synthetic regression problems and learning the solutions of elliptic PDE, which demonstrate the accuracy and efficiency of present work. 

\end{abstract}

\end{frontmatter}

\section{Introduction}

Deep neural networks have attracted extensive attention in recent times. Due to their powerful potential in approximating high-dimensional nonlinear maps, and universal approximation property to represent a rich class of functions, DNNs have been successfully employed in problems from various research areas. However, effectively training DNN is still challenging due to the difficulty of escaping local minima in non-convex optimization and avoiding overfitting in over-parameterized networks. 

Bayesian learning is appealing because of its ability to capture uncertainties in the model parameter, and MCMC sampling helps to address the overfitting issue. There has been extensive work bringing the Bayesian methods to the context of DNN optimization. The stochastic gradient Langevin dynamics (SGLD) \cite{sgld} is first proposed and becomes a popular approach in the family of stochastic gradient MCMC algorithms\cite{chen2014stochastic, ma2015complete, PSGLD}. SGLD is the first-order Euler discretization of Langevin diffusion with stationary distribution on Euclidian space. It can be viewed as adding some noise to a standard stochastic gradient optimization algorithm. Since it resembles SGD, SGLD inherits the advantage of SGD where the gradients are stochastically approximated using mini-batches. This makes MCMC scalable and provides a seamless transition between stochastic optimization and posterior sampling. It was shown that samples from SGLD will converge to samples from the true posterior distribution with annealed step size \cite{sg-mcmc-convergence, sgld}. 

In DNN, the underlying models may have complicated geometric properties and possess non-isotropic target density functions \cite{dauphin2014identifying,PSGLD, chen2014stochastic}. When the components of parameters have different curvature, generating samples using a universal step size for every model parameter may cause slow mixing and can be inefficient. In the optimization literature, there are many approaches to accelerate the gradient descent, such as preconditioning and Newton's method \cite{dauphin2015equilibrated, zhang2011quasi, byrd2016stochastic, bordes2009sgd}. However, naively borrowing this idea and using a preconditioning matrix in SGLD fails to produce a proper MCMC scheme, the Markov chain does not target the underlying posterior except for a few cases\cite{PSGLD, simsekli2016stochastic}. Considering that a Langevin diffusion with invariant measure can be directly defined on a Riemannian manifold, and the expected Fisher information is one typical choice for the Riemannian metric tensor \cite{Riemann_LD_HMC}, SGRLD is proposed \cite{SGRLD}. Built-up from Riemannian Langevin dynamics, SGRLD is a discretization of the Riemannian Langevin dynamics and the gradients are approximated stochastically. It incorporates local curvature information in parameter updating scheme, such that constant step size is adequate along with all directions.  However, the full expected Fisher information is usually intractable. A more computationally efficient preconditioner is needed to approximate second-order Hessian information. Preconditioned SGLD adopts the same preconditioner as introduced in RMSprop \cite{tieleman2012lecture} as discussed in \cite{PSGLD} which reduces the computational and storage cost. One can update the preconditioner sequentially taking into account the current gradient and previous preconditioning matrix. The preconditioner is in a diagonal form and can handle scale differences in the target density. However, the algorithm in \cite{PSGLD} introduces a permanent bias on the MSE due to ignoring a correction term in the updating equation. 

On the other hand, DNN models are usually over parameterized and require extensive storage capacity as well as a lot of computational power. The over specified models may also lead to bad generalization and large prediction variance. Enforcing sparsity in the network is necessary. In \cite{sgld-sa}, the authors propose an adaptive empirical Bayesian method for sparse learning. The idea is to incorporate an adaptive empirical Bayesian model selection techniques with SG-MCMC sampling algorithm (SGLD-SA). In SGLD-SA algorithm\cite{sgld-sa}, one adopts a spike-and-slab prior and obtains a Bayesian mixture DNN model. The model parameters are sampled from the adaptive hierarchical posterior using SG-MCMC, and the hyperparameters in the priors are optimized via stochastic approximation adaptively. The algorithm automatically searches and penalizes the low probability parameters, and identifies promising sparse high posterior probability models \cite{EMVS}. One can also apply a pruning strategy to cut off model parameters with small magnitudes to further enforce sparsity in the network \cite{molchanov2016pruning, lin2017runtime}. The performance of the sparse approach is demonstrated with numerous examples, and the method is also shown to be robust in adversarial attacks. Theoretically, the authors show that the proposed algorithm can asymptotically converge to the correct distribution.

In support of the advantages and considering the issues of the above-mentioned methods, we incorporate the preconditioned SGLD methods with sparse learning. We will apply the proposed method to learn solutions of partial differential equations with heterogeneous coefficients. Numerous approaches have been proposed to numerically solve ODEs and PDEs with deep neural networks, for example, parametric PDE \cite{Ying_paraPDE}, ODE systems driven by data \cite{NeuralODE, QinXiu2018dataODE}, time-dependent multiscale problems \cite{wang2018NLMCDNN, wang2020reduced} and physical informed DNN (\cite{PINN1, PINN2, zhang2019quantifying, PCDL_nz}). Moreover, various types of network architectures are constructed to achieve efficient learning based on existing fast numerical solvers. These approaches include designing multigrid neural networks \cite{Fan2018MNNH, he2019mgnet}, constructing multiscale models \cite{wang2018Gmsfem, wang2018NLMCDNN, wang_multiphase}, learning surrogate reduced-order models by deep convolution networks \cite{deepconv_Zabaras, E_deepRitz, cheung2020deep} and so on. 
 
This work attempts to design an efficient sparse deep learning algorithm, and apply it to learn the solutions elliptic PDE with heterogeneous coefficients. Numerical simulations for these problems are challenging since it naturally contains heterogeneities from various scales as well as uncertainties. Based on model reduction idea, for example, generalized multiscale finite element methods (GMsFEM) \cite{GMsFEM13, MixedGMsFEM, OnlineStokes}, the authors \cite{wang_multiphase} design appropriate sparse DNN structure to learn the map from the heterogeneous permeability to velocity fields in Darcy's flow. The idea is to apply locally connected/convolutional layers which can be an analogy to the upscaling and downscaling procedures in multiscale methods. However, the network is still over parameterized. In particular, the last decoding step joins neurons representing features on the coarser level to the neurons representing the fine-scale solutions and is realized by a fully connected layer. Due to the large degrees of freedom in the fine grid solution, the number of parameters in the network will be very large and result in inefficient training.
Our main contribution is to bring together preconditioned SGLD and stochastic approximation to achieve efficient and sparse learning. We propose an adaptive empirical Bayesian algorithm, where the neural network parameters are sampled from a Bayesian mixture model using PSGLD method, and the latent variables are smoothly optimized during stochastic approximation. PSGLD incorporates local curvature information in parameter updating scheme, thus it is suitable to deal with our problem which possesses multiscale nature. More importantly, we will sequentially update the preconditioning matrix under the framework of stochastic approximation, such that the bias introduced by ignoring the correction term in the sampling approaches to zero asymptotically. We theoretically show the convergence of the proposed algorithm and demonstrate its performance in several numerical experiments.

The paper is organized as follows. In Section \ref{sec:sgld}, we review some basic ideas in SGLD, SGRLD. In Section \ref{sec:sgld-sa}, the sparse adaptive empirical Bayesian approach is reviewed. Our main algorithm which combines preconditioned SGLD with sparse learning is explored in Section \ref{sec:psgld-sa}. Its convergence is discussed in Section \ref{sec:analysis}. Applying the proposed method to a large-p-small-n regression problem, and to learn solutions of elliptic problems with heterogeneous coefficients, its performances are presented in Section \ref{sec:numerical}. A conclusion is made in the last Section \ref{sec:conclusion}.

\section{Stochastic gradient Langevin dynamics (SGLD) and stochastic gradient Riemann Langevin dynamics(SGRLD) } \label{sec:sgld}
Throughout the paper, we denote by $\boldsymbol \beta$ the model parameters with $p(\boldsymbol \beta)$ as a prior distribution, and $D = \{d_i \}_{i=1}^N$ the entire dataset, where $d_i = (x_i, y_i)$ is an input-output pair for the model. Let $p(d|\boldsymbol \beta)$ be the likelihood, the posterior is then $p(\boldsymbol \beta|D) \propto p(\boldsymbol \beta) \prod_{i=1}^{N} p(d_i|\boldsymbol \beta)$. SGLD combines the idea from stochastic gradient algorithms and posterior Bayesian sampling using Langevin dynamics. The loss gradient is approximated efficiently use mini-batches of data in SGLD, and the uncertainties in the model parameter can be captured through Bayesian learning to avoid overfitting. The model parameters update as follows:
\begin{equation*}
{\boldsymbol \beta}_{k+1} = {\boldsymbol \beta}_{k} + \epsilon_k \nabla_{\boldsymbol \beta} \tilde{L} ({\boldsymbol \beta}_{k} ) + \mathcal{N}(0, 2\epsilon_k \tau^{-1})
\end{equation*}
where for a subset of $n$ data points $d_k = \{d_{k1}, \cdots, d_{kn} \}$
\begin{equation*}
\nabla_{\boldsymbol \beta} \tilde{L} (\boldsymbol \beta ) = \nabla_{\boldsymbol \beta} \log p(\boldsymbol \beta) + \frac{N}{n} \sum_{i=1}^n \nabla_{\boldsymbol \beta} \log p(d_{ki}|\boldsymbol \beta)
\end{equation*} 
is the stochastic gradient computed using a mini-batch, which is used to approximate the true gradient $\nabla_{\boldsymbol \beta} {L} (\boldsymbol {\beta})$.

However, if the components of the model parameter $\boldsymbol \beta$ possess different scales, the invariant probability distribution for the Langevin equation is not isotropic, using standard Euclidian distance may lead to slow mixing. Stochastic Gradient Riemann Langevin Dynamics (SGRLD) \cite{SGRLD} is a generalization of SGLD on a Riemannian manifold. In this case, consider the probability models on a Riemann manifold with some metric tensor $G^{-1}(\boldsymbol \beta)$, the parameter updates can be guided using the geometric information of this manifold as follows:
\begin{equation} \label{eq:sgrld}
{\boldsymbol \beta}_{k+1} = {\boldsymbol \beta}_{k} + \epsilon_k  \left[ G(\boldsymbol \beta_k ) \nabla_{\boldsymbol \beta} \tilde{L} ({\boldsymbol \beta}_{k} ) + \Gamma({\boldsymbol \beta}_k)  \right] +  \mathcal{N}(0, 2\epsilon_k \tau^{-1}G(\boldsymbol \beta_k )  )
\end{equation}
where $\Gamma({\boldsymbol \beta}_k)$ is an additional drift term and $\Gamma_i({\boldsymbol \beta}_k) = \sum_j \frac{\partial G_{ij}(\boldsymbol \beta_k)  }{\partial \beta_j}$. The he expected Fisher information can be used as a natural metric tensor, however it is intractable in many cases. One can choose a more practical metric tensor and use it as a preconditioning matrix. 



\section{SGLD with stochastic approximation (SGLD-SA)}\label{sec:sgld-sa}

In order to achieve sparse learning in DNN, in \cite{sgld-sa}, the authors propose an adaptive empirical Bayesian method. It assumes that the weight parameters $\beta_{lj}$, the $j$-th neuron in the $l$-th layer, follows spike-and-slab Gaussian Laplace prior
\begin{equation*}
\pi (\beta_{lj} | \sigma^2, \gamma_{lj})  = (1-\gamma_{lj})  \mathcal{L}_p(0, \sigma v_0) + \gamma_{lj}  \mathcal{N}(0, \sigma^2 v_1) 
\end{equation*}
where $\gamma_{lj} \in \{0, 1\}$ are the latent binary variable selection indicators, $\mathcal{L}_p$ is the Laplace distribution, and $\mathcal{N}$ is the Normal distribution. The error variance $\sigma^2$ follows an inverse gamma prior $\pi(\sigma^2) = IG(\nu/2, \nu \lambda/2)$. The prior for $\gamma$ follows a Bernoulli distribution, $\pi(\gamma_l| \delta_l) = \delta_l^{|\gamma_l|} (1-\delta_l)^{p_l - |\gamma_l|}$, which incorporate uncertainty regarding which variables $\beta_{lj}$ need be included in the model. Here, $|\gamma_l| = \sum_{j}\gamma_{lj}$, and $\delta_l$ follows  $\pi(\delta_l) = \delta_l^{a-1} (1-\delta_l)^{b-1 }$ where $a$, $b$ are some positive constants.. 

Let $d^m$ be the $m$-th mini-batch of the dataset. The likelihood for a regression problem can be rewritten as 
\begin{equation*}
\pi (d^m| \boldsymbol \beta, \sigma^2) = \frac{1}{(2\pi \sigma^2)^{n/2}} \exp \big\{-\frac{  \sum \limits_{x^m_i \in d^m} (x^m_i -\mathcal{F}(x^m_i; \boldsymbol \beta)  )^2 }{ 2 \sigma^2 } \big\}
\end{equation*}
where $\mathcal{F}$ denotes a map describing the input-output relationship from $x_i^m$ to $y_i^m$.

Then, the posterior follows
\begin{equation} \label{eq:post_dist}
\pi (\boldsymbol \beta, \sigma^2, \delta, \gamma | d^m) \propto \pi (d^m| \boldsymbol \beta, \sigma^2)^{\frac{N}{n}}  \pi (\boldsymbol \beta | \sigma^2, \gamma)  \pi(\sigma^2| \gamma) \pi(\gamma| \delta) \pi(\delta)
\end{equation}

Now treat $\gamma$ as ``missing data". At iteration $k$, 
instead of sampling from true posterior with respect to the whole dataset $\mathcal{D}$, one needs to sample from $Q$ with respect to a mini-batch $\mathcal{B}$ 
\begin{equation*}
Q(\boldsymbol \beta, \sigma, \delta |\beta_k, \sigma_k, \delta_k) = \mathbb{E}_{\mathcal{B}} \left[ \mathbb{E}_{\mathcal{\gamma|\mathcal{D}}} [\log \pi (\beta, \sigma, \delta, \gamma)|\mathcal{B}]\right]  
\end{equation*}
and it can be separated as
\begin{equation*}
Q(\boldsymbol \beta, \sigma, \delta |\beta_k, \sigma_k, \delta_k) = Q_1 (\boldsymbol \beta, \sigma |\boldsymbol {\beta}_k, \sigma_k, \delta_k) + Q_2(\delta |\boldsymbol {\beta}_k, \sigma_k, \delta_k) + C
\end{equation*}
where
\begin{align*}
Q_1 (\boldsymbol \beta, \sigma |\boldsymbol {\beta}_k, \sigma_k, \delta_k) &= \frac{N}{n} \log \pi(d^m| \boldsymbol \beta) - \sum_
{l \in L_D} \sum_{j\in p_l} \frac{\beta_{lj}^2}{2 \sigma_0^2} - \frac{p+\nu+2}{2} \log(\sigma^2) - \\
&\sum_{l \in L_S} \sum_{j\in p_l} \big\{ \frac{ |\beta_{lj}| }{\sigma} E[  \frac{1}{v_0 (1-\gamma_{lj} )} ] +\frac{ \beta_{lj}^2 }{2\sigma^2} E[  \frac{1}{v_1  \gamma_{lj} } ]   \big\} -\frac{\nu \lambda}{2\sigma^2 } 
\end{align*}

\begin{equation*}
Q_2(\delta |\boldsymbol {\beta}_k, \sigma_k, \delta_k) = \sum_{l \in L_S} \sum_{j\in p_l} \log (\frac{\delta_l}{1 -\delta_l}) E[ \gamma_{lj}] + (a-1) \log(\delta_l)+(p_l+b-1) \log(1-\delta_l)
\end{equation*}
where $L_S$ denotes sparse layers, and $L_D$ denotes non-sparse layers.

 The adaptive empirical Bayesian algorithm samples $\boldsymbol \beta$ from $Q$ and iteratively optimize $Q$ with respect to $\sigma^2, \gamma, \delta$ via stochastic approximation as in Algorithm \ref{alg:SGLD-SA}.
\begin{algorithm} [!htb]
	\caption{SGLD-SA}\label{alg:SGLD-SA}
	\begin{algorithmic}[1] 
		\Input{Initialize $ \beta_1, \rho_1, \kappa_1, \delta_1, \sigma_1$. Given target sparse rate $s$, step size $\omega_k$ }
		\sForAll{ $k \gets 1: \#iterations$} {
		\State $\displaystyle{ \boldsymbol {\beta}_{k+1} \gets  \boldsymbol {\beta}_k +  \epsilon_k \nabla_{ \boldsymbol \beta} Q(\cdot|d_k ) + \mathcal{N}(0, 2\epsilon_k \tau^{-1})}$
		\State $\displaystyle{a_{lj} \gets \pi(\boldsymbol \beta_k^{lj} | \gamma_{lj}=1 )\delta_l^k }$, $\displaystyle{b_{lj} \gets \pi(\boldsymbol \beta_k^{lj}  | \gamma_{lj}=0)(1-\delta_l^k)}$
		\State $\displaystyle{\rho_{k+1} \gets (1-\omega_{k+1}) \rho_k +  \omega_{k+1} \frac{a}{a +b}}$
		\State $\displaystyle{\kappa_{k+1,0} \gets (1-\omega_{k+1}) \kappa_{k,0} +  \omega_{k+1} \frac{1- \rho_{k+1}}{v_0}}$
		\State $\displaystyle{\kappa_{k+1,1} \gets (1-\omega_{k+1}) \kappa_{k,1} +  \omega_{k+1} \frac{\rho_{k+1}}{v_1}}$
		\State $\displaystyle{\sigma_{k+1} \gets (1-\omega_{k+1}) \sigma_{k} +  \omega_{k+1} R}$ 
		\State $ \displaystyle{\delta_{k+1} \gets (1-\omega_{k+1} ) \delta_{k} +  \omega_{k+1} \frac{\sum_j \rho_{k+1}^j +a-1}{a+b+p-2} }$ 
		\If {Pruning} 
			\State Prune the last $s\%$ weights with smallest magnitude
			\State Increase the sparse rate	
		\EndIf
	    }
	\end{algorithmic}
\end{algorithm}
We note that the update formula of latent variables $\rho, \kappa, \delta, \sigma$ are motivated by EM approach to Bayesian variable selection (EMVS) \cite{rovckova2014emvs}. In Algorithm \ref{alg:SGLD-SA}, $\rho_{lj} = E[\gamma_{lj}]$, $\omega_k$ is the step size in updating latent variables, $ \kappa_{k,0} = E[  \frac{1}{v_0 (1-\gamma_{lj} )}]$ and $ \kappa_{k,1} = E[  \frac{1}{v_1  \gamma_{lj} }]$, $R$ is the positive root to the following quadratic formula:
\begin{align*}
&\big \{ N+\sum_{l\in L_s} p_l +\nu \big\} \sigma^2 + \big \{  ||\sum_{l\in L_s} \kappa_{k,0}^l \circ \beta_{k+1}^l ||_1  \big\} \sigma \\
&+  \big \{ \frac{N}{n} \sum_{x^m_i \in d^m} (y_i^m-\mathcal{F}(x_i^m; \boldsymbol \beta)  )^2 + ||\sum_{l\in L_s} \kappa_{k,1}^l \circ \beta_{k+1}^l ||_2^2 + \nu \lambda \big\}  = 0  
\end{align*}

where $\circ$ denotes the point-wise product, $||\cdot||_1$ and $||\cdot||_2$ are the vector $l_1$ and $l_2$ norm correspondingly.

\section{Preconditioned SGLD with stochastic approximation (PSGLD-SA)  } \label{sec:psgld-sa}

As seen in Section \ref{sec:sgld}, all model parameters $\boldsymbol \beta$ are updated using the same learning rate $\epsilon_k$, this may cause slow mixing if the loss function has very different scales in different directions, and a small enough learning rate is required to avoid divergence in the largest positive curvature direction. 

Here, we will introduce a preconditioning matrix $G(\boldsymbol \beta)$ to guide the updating directions during sampling. In gradient descent algorithms, the optimization can be improved using the second order information, i.e. the inverse of the Hessian matrix, as the preconditioning matrix. However,  it is too computationally expensive to store
and invert the full Hessian during the training.  An efficient approximation is to use the same preconditioner as in RMSprop \cite{tieleman2012lecture}. The idea is to scale the gradient using a moving average of its recent norm in each iteration, so that one can adapt the step size separately for each weight. By keep a moving average for each weight parameter from the previous step, one can control the changes among adjacent mini batches. We propose a sequentially updated preconditioner using the stochastic approximation idea as follows
\begin{align}\label{eq:preconditioner}
 G(\boldsymbol {\beta}_k) &=  diag^{-1} (\eta + \sqrt{V(\boldsymbol {\beta}_k)}) \\
 V(\boldsymbol {\beta}_k) &=  \alpha_k V(\boldsymbol {\beta}_{k-1}) + (1-\alpha_k)  g(\boldsymbol {\beta}_k) \circ g(\boldsymbol {\beta}_k)
\end{align}
where $\eta$ is a regularization constant, and $\alpha_k = (1-\omega_k)$, $g(\boldsymbol {\beta}_k) = \nabla_{\boldsymbol {\beta}} Q$. Importantly, we note that the weight parameter $\alpha_k$ is a sequence approaching $1$ as the time step $k$ increases, which is different from the constant $\alpha$ in \cite{PSGLD}. The change in the parameters will then be
\begin{equation}\label{eq:psgld-update}
   \displaystyle{\triangle \boldsymbol {\beta}_k = \epsilon_k \big(G(\boldsymbol {\beta}_k) g(\boldsymbol {\beta}_k)  + \Gamma({\boldsymbol \beta}_k) \big)   + \mathcal{N}(0, 2\epsilon_k \tau^{-1}G^{\frac{1}{2}}(\boldsymbol {\beta}_k) )} 
\end{equation}
where $\Gamma_i({\boldsymbol \beta}_k) = \sum_j \frac{\partial G_{ij}( \boldsymbol {\beta}_k)  }{\partial \beta_j}$.

 We note that in \cite{PSGLD}, $\Gamma({\boldsymbol \beta}_k)$ is ignored in practice, and $\alpha$ is a constant. This produces a permanent bias $\mathcal{O} \left( \frac{(1-\alpha)^2}{\alpha^3}  \right)$ on the MSE. To address this issue, we let $\alpha_k$ gradually approach $1$ during the adaptive optimization of the latent variables, then the bias mentioned before will decrease. 
 To be specific, we have 
\begin{align*}
\left| \sum_{k =1}^K \Gamma_i(\boldsymbol {\beta}_k ) \right|&= \left| \sum_{t =1}^T (1-\alpha_k) V^{-\frac{3}{2}}(\boldsymbol{ \beta}_k)  g(\boldsymbol{ \beta}_k) \frac{ \partial g(\boldsymbol{ \beta}_k)}{\partial \boldsymbol{ \beta} }   \right| \\
&=  \left| \sum_{k =1}^K   (1-\alpha_k) g(\boldsymbol {\beta}_k) { \left[ \alpha_{k-1} V(\boldsymbol {\beta}_{k-1}) + (1-\alpha_{k-1}) g(\boldsymbol {\beta}_{k-1})^2 \right] ^{-\frac{3}{2}}}   \frac{ \partial g(\boldsymbol{\beta}_k)}{\partial \boldsymbol{\beta} }   \right| \\
& \lesssim \left| \sum_{k =1}^K  (1-\alpha_k) \frac{g(\boldsymbol{\beta}_k)}{  \alpha_{k-1}^{\frac{3}{2}} V(\boldsymbol{\beta}_{k-1})^{\frac{3}{2}}   }  \frac{ \partial g({ \beta}_k)}{\partial \boldsymbol{\beta} }    \right| \\
& \lesssim \left| \sum_{k=1}^K  (1-\alpha_k) \frac{g(\boldsymbol{\beta}_k)}{ \alpha_1^{\frac{3}{2}}  V^{\frac{3}{2}}(\boldsymbol{ \beta}_{k-1})   }  \frac{ \partial g(\boldsymbol{\beta}_k)}{\partial { \boldsymbol\beta} }    \right| 
\end{align*} 

Then we have
\begin{equation}\label{eq:gamma_bound}
    \left| \sum_{k =1}^K  \Gamma_i(\boldsymbol{\beta}_k) \right| \lesssim M \left| \sum_{k=1}^K  \frac{(1-\alpha_k)}{  \alpha_1^{\frac{3}{2}}  }    \right|
\end{equation}
due to the assumption that the derivative of the gradients are bounded, $ |V^{-\frac{3}{2}}(\boldsymbol{ \beta}_{k-1}) g(\boldsymbol{\beta}_k)\frac{ \partial g(\boldsymbol{\beta}_k)}{\partial\boldsymbol { \beta} } | \leq M$ for some constant $M>0$.

Typically, let $\alpha_k$ be in the form of $\alpha_k = 1- c_1(c_2+k)^{-\gamma}$ for some $\gamma \in (0.5, 1]$, and constants $c_1, c_1$, we can see that the bias introduced $\sum_{k=1}^K \frac{(1-\alpha_k)^2}{ \alpha_1^{3}}  $ on the MSE will approach $0$ as $K \rightarrow \infty$.

Thus, our proposed adaptive preconditioned SGLD samples $\boldsymbol \beta$ and optimizes $\sigma^2, \gamma, \delta$ as in Algorithm\ref{alg:PSGLD-SA}.
\begin{algorithm} [!htb]
	\caption{PSGLD-SA}\label{alg:PSGLD-SA}
	\begin{algorithmic}[1] 
		\Input{Initialize $\boldsymbol{\beta}_1, \rho_1, \kappa_1, \delta_1, V_1$, let $\alpha_1 = 0.9, \eta =10^{-3}$}
		\sForAll{ $k \gets 1: \#iterations$} {
			\State $g(\boldsymbol{\beta}_{k}) \gets \nabla_{\boldsymbol \beta} Q(\cdot|d_k )$
			\If {$k == 1$}
				\State $V(\boldsymbol{\beta}_{k}) \gets g(\boldsymbol{\beta}_{k}) \circ g(\boldsymbol{\beta}_{k}) $
			\Else
				\State $V(\boldsymbol{\beta}_{k}) \gets (1-\alpha_k) V(\boldsymbol{\beta}_{k-1}) + \alpha_k g(\boldsymbol{\beta}_{k}) \circ g(\boldsymbol{\beta}_{k}) $
			\EndIf
			\State $G(\boldsymbol{\beta}_{k}) \gets diag^{-1} (\eta + \sqrt{V(\boldsymbol{\beta}_{k})}) $
			\State $\displaystyle{\boldsymbol{\beta}_{k+1} \gets \boldsymbol{\beta}_{k} + \epsilon_k \big(G(\boldsymbol{\beta}_{k}) g(\boldsymbol{\beta}_{k} ) \big)   + G^{\frac{1}{2}}(\boldsymbol{\beta}_{k}))\mathcal{N}(0, 2\epsilon_k \tau^{-1})}$
			\State Updating hyperparameters by running steps 3-11 in Algorithm \ref{alg:SGLD-SA}}
	\end{algorithmic}
\end{algorithm}

\section{Convergence results}\label{sec:analysis}
Now, we will discuss the weak convergence of our proposed algorithm PSGLD-SA. First, we will take a look at the hyperparameters.
 Denote by $\boldsymbol \theta$ all the hyperparameters $(\rho, \kappa, \sigma, \delta)$. The stochastic approximation attempts to get the optimal $\boldsymbol \theta_*$ based on the asymptotically target distribution $\pi(\boldsymbol{\beta},\boldsymbol \theta_*)$. Define $H(\boldsymbol \theta,\boldsymbol{\beta}) = g_{\boldsymbol \theta}(\boldsymbol{\beta})- \boldsymbol \theta$, where $g_{\boldsymbol {\theta}}(\boldsymbol{\beta})$ represents a function to obtain optimal $\boldsymbol {\theta}$ given current model parameters $\boldsymbol{\beta}$. Denote by its mean field function $h(\boldsymbol \theta) = \mathbb{E} [H(\boldsymbol \theta, \boldsymbol{\beta})]$. SA aims to solve the fixed point equation $\displaystyle{ \int g_{\boldsymbol \theta}(\boldsymbol{\beta}) \pi(\boldsymbol{\beta},\boldsymbol \theta) d \boldsymbol{\beta} = \boldsymbol \theta}$, which is to find the root $\boldsymbol \theta_*$ of the equation $h(\boldsymbol \theta) = 0$. 
 As described in Algorithm \ref{alg:PSGLD-SA}, in each iteration, we first sample $\boldsymbol \beta_{k+1}$ using precontioned SGLD based on $\boldsymbol \theta_k$, then update the latent variables  using 
 \begin{equation*}
     \boldsymbol \theta_{k+1} = \boldsymbol \theta_k + \omega_{k+1} H(\boldsymbol \theta_k,\boldsymbol{\beta_{k+1}}),
 \end{equation*}
 where the map $g$ is motivated by EMVS. However, we only use a small set of data of $n$ samples instead of the full set in the computation of obtaining optimal latent variables. This will result a bias $\Delta(n,\boldsymbol\theta_i, \boldsymbol\beta_{i+1})$ at each step. That is, we actually use $ \boldsymbol \theta_{k+1} = \boldsymbol \theta_k + \omega_{k+1} \tilde{H}(\boldsymbol \theta_k,\boldsymbol{\beta_{k+1}})$ with
\begin{equation}\label{eq:H_tilde}
\tilde{H}(\boldsymbol \beta, \boldsymbol \theta) =H(\boldsymbol \beta, \boldsymbol \theta)  + \Delta(n,\boldsymbol\theta_i, \boldsymbol\beta_{i+1}),
\end{equation}
and we assume $ \mathbb{E}||\Delta(n,\boldsymbol\theta_i, \boldsymbol\beta_{i+1})||^2 \leq C^2$ for some constant $C$.

 Following a similar proof in \cite{sgld-sa}, under suitable assumptions, the adaptive empirical Bayesian method for sparse approximation algorithm has the following convergence results. The details of the proof are in Appendix \ref{app:theta}.
 \begin{theorem}
For a sufficiently large $k_0$, there exists a constant $\lambda$ such that 
\begin{equation*}
\mathbb{E} \left[ ||\boldsymbol\theta_k - \boldsymbol\theta^* ||^2 \right] = \mathcal{O} (\lambda \omega_k + \sup_{i\geq k_0}\mathbb{E} ||\Delta(n,\boldsymbol\theta_i, \boldsymbol\beta_{i+1})||).
\end{equation*} 
 \end{theorem}

Next, we present a weak convergence result of the model parameters. 
\begin{corollary}
Under Assumptions 2 in \cite{sg-mcmc-convergence}, the bias and MSE of PSGLD-SA for $K$ steps with decreasing step size $\epsilon_k$ is bounded,
the distribution of $\boldsymbol \beta_k$ converges weakly to the target posterior with a controllable bias, as $\epsilon_k\rightarrow 0$ and $k\rightarrow \infty$.
\end{corollary}
\begin{proof}
With geometric information for probability models, the Langevin diffusion on the manifold is described by
\begin{equation}\label{eq:langevin_manifold}
d \boldsymbol {\beta (t)} =    G(\boldsymbol {\beta } (t))  \nabla_{\boldsymbol \beta} {L} (\boldsymbol {\beta }  (t), \boldsymbol\theta^* ) + \Gamma(\boldsymbol {\beta } (t))  +  G^{\frac{1}{2}} (\boldsymbol {\beta } (t)) d \mathcal{B}_t
\end{equation}
where $\mathcal{B}_t$ is the Brownian motion.

Denote by $\mathcal{L}$ the generator for \eqref{eq:langevin_manifold}, then
\begin{equation}\label{eq:cont_generator}
\mathcal{L} =  \left[G(\boldsymbol {\beta_k} ) \nabla_{\boldsymbol \beta} {L} (\boldsymbol {\beta_k } ,\boldsymbol\theta^*  ) + \Gamma(\boldsymbol {\beta_k } )  \right]  \cdot \nabla_{\boldsymbol \beta}  + 2 G^{\frac{1}{2}}(\boldsymbol {\beta}) G^{\frac{1}{2}}(\boldsymbol {\beta_k })^T : \nabla_{\boldsymbol \beta_k} \nabla_{\boldsymbol \beta}^T
\end{equation}
The generator $\mathcal{L}$ is associated with the backward Kolmogorov equation 
\begin{equation*}
\mathbb{E} [\phi(\boldsymbol {\beta_k } )] = e^{t\mathcal{L}} \phi(\boldsymbol {\beta}_0 )
\end{equation*}

 In PSGLD-SA, one will sample from the adaptive hierarchical posterior using \eqref{eq:preconditioner} \eqref{eq:psgld-update}, and gradually optimize the latent variables through stochastic approximation.

Write the local generator of our proposed algorithm as
\begin{equation}
\tilde{\mathcal{L}}_k = \left[G(\boldsymbol {\beta_k} )  \tilde{g}_k \right]    \cdot \nabla_{\boldsymbol \beta}  + 2 G^{\frac{1}{2}}(\boldsymbol {\beta }) G^{\frac{1}{2}}(\boldsymbol {\beta_k })^T : \nabla_{\boldsymbol \beta_k} \nabla_{\boldsymbol \beta}^T
\end{equation}
where $\tilde{\mathcal{L}}_k  =  \mathcal{L} + \Delta V_k$, and 
\begin{equation*}
\Delta V_k = \left[ G(\boldsymbol {\beta_k} ) \left( \nabla_{\boldsymbol \beta} {L} (\boldsymbol {\beta_k},\boldsymbol \theta^*) -\tilde{g}_k \right) + \Gamma(\boldsymbol {\beta_k} ) \right] \cdot \nabla_{\boldsymbol \beta}.
\end{equation*}
Thus 
\begin{equation*}
\tilde{g}_k =  \nabla_{\boldsymbol \beta} {L} (\boldsymbol {\beta}_k  ) + \xi_k + \mathcal{O}(k^{-\gamma}+ \sup_{i\geq k_0}\mathbb{E} ||\Delta(n,\boldsymbol\theta_i, \boldsymbol\beta_{i+1})||)
\end{equation*}
where $\xi_k$ is a random vector denoting the difference between the true gradient and stochastic gradient, and $\mathcal{O}(k^{-\gamma}+\sup_{i\geq k_0}\mathbb{E} ||\Delta(n,\boldsymbol\theta_i, \boldsymbol\beta_{i+1})||)$ is the bias term generated by SA.

Given a test function $\phi$ of interest, let $\bar{\phi}$ be the posterior average of $\phi$ under the invariant measure of the SDE \eqref{eq:langevin_manifold}. Let $\boldsymbol {\beta}_{k}$ be the numerical samples, and define 
$\hat{\phi} = \sum_{k=1}^{K} \frac{\epsilon_k}{S_K} \phi(\boldsymbol {\beta}_{k} )$, where $S_K = \sum_{k=1}^{K} \epsilon_k$. Let $\psi$ be a functional which solves the Poisson equation 
\begin{equation*}
\mathcal{L} \psi(\boldsymbol {\beta}_{k}) = \phi(\boldsymbol {\beta}_{k}) - \bar{\phi}.
\end{equation*}

Following a similar proof as in \cite{sg-mcmc-convergence}, one can obtain the following results.
The bias of PSGLD-SA is
\begin{equation*}
   |\mathbb{E} \hat{\phi} - \bar{\phi} | \leq \frac{1}{S_K} |\mathbb{E}   \psi(\boldsymbol {\beta}_{K} ) - \psi(\boldsymbol {\beta}_{0} )| +  \sum_{k=1}^{K} \frac{\epsilon_k}{S_K} \mathbb{E} ||\Delta V_k \psi(\boldsymbol {\beta}_{k-1} ) || + C \sum_{k=1}^{K} \epsilon_k^2
\end{equation*}

Formally, we note that in the above bound for the bias, the term $\sum_{k=1}^{K} \frac{\epsilon_k}{S_K} \mathbb{E} ||\Delta V_k \psi(\boldsymbol {\beta}_{k-1} )||$ is important. It is related to the bias introduced by stochastic approximation and ignoring $\Gamma(\boldsymbol {\beta}_k)$.
By Assumptions 2 in \cite{sg-mcmc-convergence} on the smootheness and boundedness on the functional $\psi$, and the boundedness of the preconditioner, it is easy to see that the bias introduced by stochastic approximation can be decomposed into (1) the the term $\displaystyle{ \sum_{k=1}^K \frac{\epsilon_k k^{-\gamma}}{S_K}}$ in the bias, which approaches 0 as $K \rightarrow \infty$, and (2) $\displaystyle{ \sum_{k=1}^K \frac{\epsilon_k }{S_K}\sup_{i\geq k_0}\mathbb{E} ||\Delta(n,\boldsymbol\theta_i, \boldsymbol\beta_{i+1})||} = \sup_{i\geq k_0}\mathbb{E} ||\Delta(n,\boldsymbol\theta_i, \boldsymbol\beta_{i+1})|| $ which is a controllable bias. The bias introduced by ignoring $\Gamma(\boldsymbol {\beta}(t) )$ can be bounded by $\displaystyle{\sum_{k=1}^K  (1-\alpha_k) \alpha_1^{-\frac{3}{2}} = \mathcal{O}(\sum_{k=1}^K  k^{-\gamma}) }$ according \eqref{eq:gamma_bound}, which goes to 0 as $K \rightarrow \infty$. 

The MSE of PSGLD-SA can be bounded by
\begin{equation*}
     \mathbb{E} (\hat{\phi} - \bar{\phi})^2 \leq  C \left( \sum_{k=1}^{K} \frac{1}{S_K^2} +  \sum_{k=1}^{K} \frac{\epsilon_k^2}{S_K^2}  \mathbb{E} ||\Delta V_k \psi(\boldsymbol {\beta}_{k-1} )||^2 + \frac{(\sum_{k=1}^{K} \epsilon_k)^2}{S_K^2} \right) 
\end{equation*}
which converges as long as $\sup_k  \mathbb{E} ||\Delta V_k \psi(\boldsymbol {\beta}_{k-1} )||^2$ is bounded. 

Thus we conclude that, as $\epsilon_k\rightarrow 0$ and $k\rightarrow \infty$, the distribution of $\boldsymbol \beta_k$ converges weakly to the target posterior with a controllable bias. The bias is expected to decrease if we enlarge the mini-batch size to approximate the gradient.

\end{proof}

\section{Numerical examples}\label{sec:numerical}
\subsection{Small $n$ large $p$ problem} \label{sec:num_np}

We first test on a linear regression problem, where the model parameters $\boldsymbol\beta \in \mathbb{R}^p$, and predictors $X \in  \mathbb{R}^{n\times p}$. We take a dataset with $n = 100$ observations
and $p = 200$ predictors. $\beta_1 = 3, \beta_2 = 1, \beta_j = 0$, for $j=1, \cdots, p$. 

For the first test (section \ref{sec:num_np} test 1), we use $\mathcal{N}_p(0, \Sigma)$ with $\Sigma_{ij} = 0.6^{|i-j|}$ to simulate predictor values $X$. The responses $y = X \boldsymbol \beta + \epsilon$, and $\epsilon \sim \mathcal{N}_n(0, 3I_n)$. The hyperparameters used for SGLD-SA are: $v_0 = 10$, $v_1= 0.1$, $\delta = 0.5$, $b=p$, $a=1$, $\lambda=1$, $\nu=1$. The hyperparameters used for PSGLD-SA are: $v_0 = 100$, $v_1= 0.1$, $\delta = 0.5$, $b=p$, $a=1$, $\lambda=1$, $\nu=1$, $\alpha =0.999$. The learning rate is $\epsilon_k = 0.05\times k^{-\frac{1}{3}}$, and the step size to update latent variables is $\omega_k = 100\times (k+100)^{-0.7}$.
The performance of SGLD-SA, PSGLD and PSGLD-SA are compared and presented in Figure \ref{fig:example1}. It shows that both SGLD-SA and PSGLD-SA work similarly well for this setting. The variance in model parameters $\beta_1$ and $\beta_2$ are similar, and they can be quantified correctly in both settings. However, without stochastic approximation, vanilla PSGLD cannot capture the uncertainties propoerly. However, Figure \ref{fig:example1} (c)-(d) show that preconditioned methods converge a little bit faster according to the testing curves. 

In the second test (section \ref{sec:num_np} test 2), we change the predictors $X$ described in test 1 a little bit. That is, we multiply $0.3$ on the first column of the predictor ($X[,1]$), to create different scales in the predictor. The response values $y$ and true regression coefficients are set to be similar to before. In this case, the uncertainties in the posterior estimation $\beta_1$ and $\beta_2$ will have different scales. We also compare the performance of SGLD-SA PSGLD, and PSGLD-SA and present them in Figure \ref{fig:example2}. In this example, we see that PSGLD-SA outperforms PSGLD and SGLD-SA the standard approach obviously.


\begin{figure}[!ht] 
		\centering
	\begin{subfigure}{.3\textwidth}
		\centering
		\includegraphics[scale=0.25]{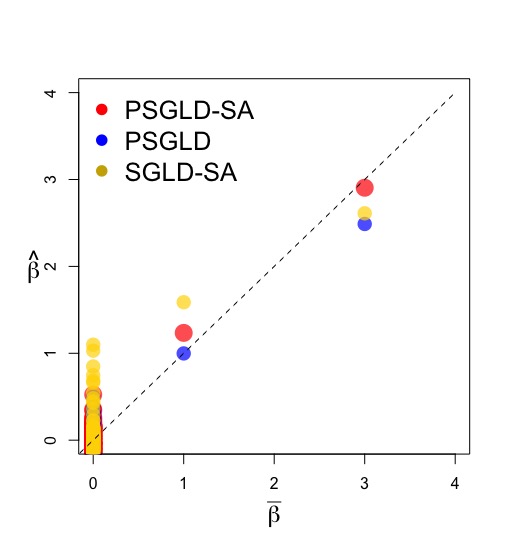}
		\caption{Posterior mean ($\bar{\boldsymbol \beta}$) vs true ($\hat{\boldsymbol \beta}$)}
	\end{subfigure}\;\;\;\;\;\;\;\;\;\;\;\;\;\;\;\;\;\;
	\begin{subfigure}{.3\textwidth}
	\centering
	\includegraphics[scale=0.25]{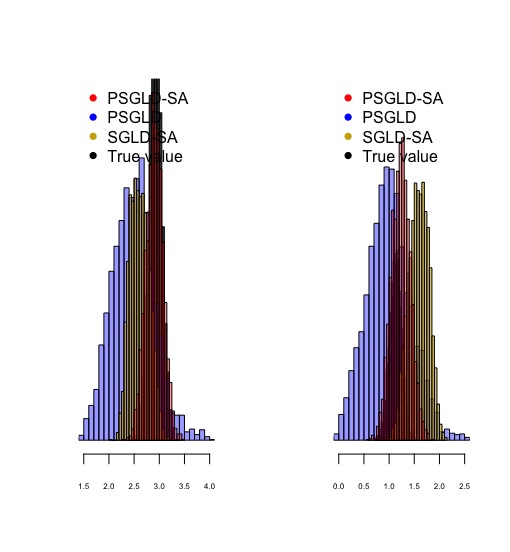}
	\caption{Posterior estimation of $\beta_1$ and $\beta_2$}
\end{subfigure}

\begin{subfigure}{.3\textwidth}
	\centering
	\includegraphics[scale=0.25]{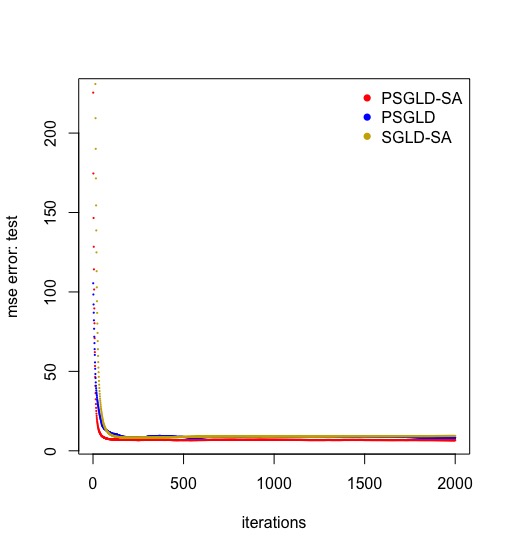}
	\caption{Testing MSE error history}
\end{subfigure}\;\;\;\;\;\;\;\;\;\;\;\;\;\;\;\;\;\;
\begin{subfigure}{.3\textwidth}
	\centering
	\includegraphics[scale=0.25]{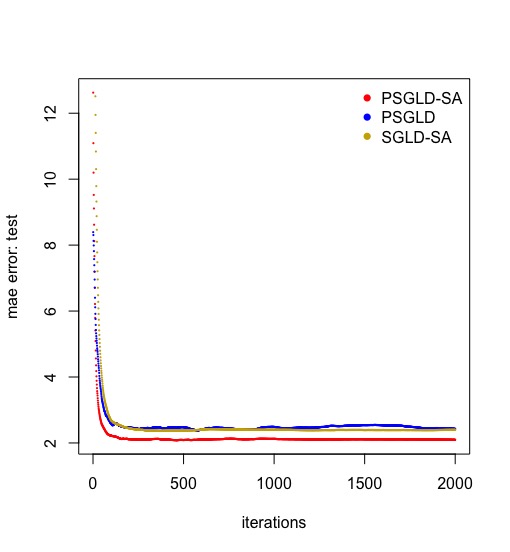}
	\caption{Testing MAE error history}
\end{subfigure}\hfill
\caption{Section \ref{sec:num_np} test 1. Large $p$ small $n$ regression for predictors with uniform scale.} \label{fig:example1}
\end{figure}


\begin{figure}[!ht]
		\centering
	\begin{subfigure}{.3\textwidth}
		\centering
		\includegraphics[scale=0.25]{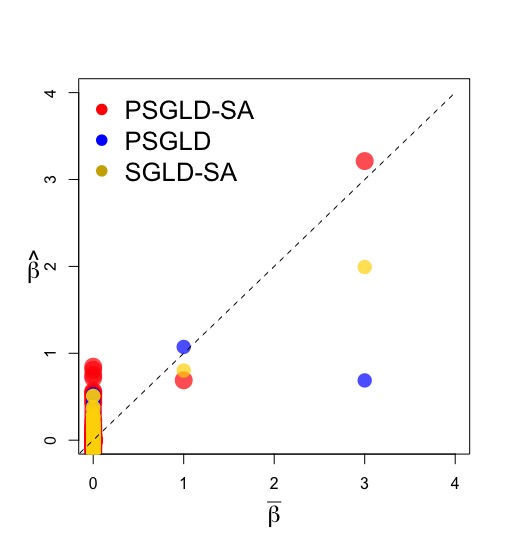}
		\caption{Posterior mean ($\bar{\boldsymbol \beta}$) vs true ($\hat{\boldsymbol \beta}$)}
	\end{subfigure}\;\;\;\;\;\;\;\;\;\;\;\;\;\;\;\;\;\;
	\begin{subfigure}{.3\textwidth}
		\centering
		\includegraphics[scale=0.25]{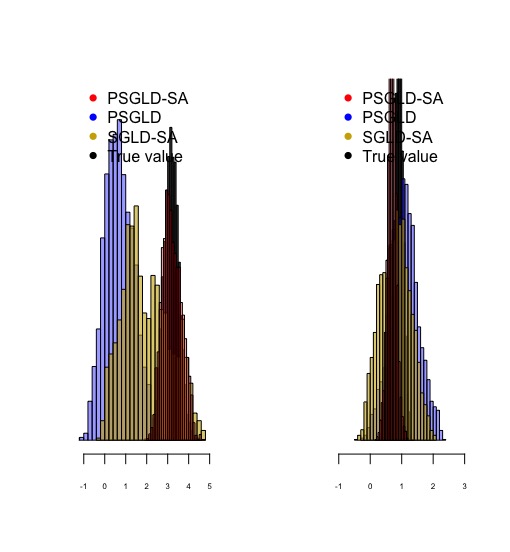}
		\caption{Posterior estimation of $\beta_1$ and $\beta_2$}
	\end{subfigure}
	
	\begin{subfigure}{.3\textwidth}
		\centering
		\includegraphics[scale=0.25]{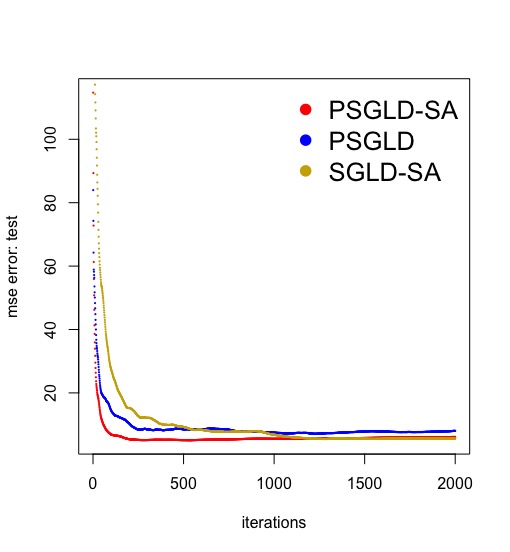}
		\caption{Testing MSE error history}
	\end{subfigure}\;\;\;\;\;\;\;\;\;\;\;\;\;\;\;\;\;\;
	\begin{subfigure}{.3\textwidth}
		\centering
		\includegraphics[scale=0.25]{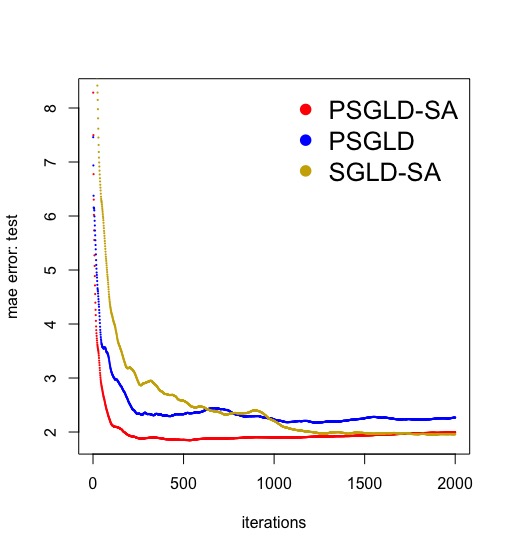}
		\caption{Testing MAE error history}
	\end{subfigure}\hfill
	\caption{Section \ref{sec:num_np} test 2. Large $p$ small $n$ regression for predictors with different scales. }\label{fig:example2}
\end{figure}

\subsection{Elliptic problem with heterogeneous coefficients}\label{sec:num_ex2}
Next, we apply the proposed approaches to solve the elliptic problem with heterogeneous coefficients. The mixed formulation of the elliptic problem reads:
\begin{align*}
\kappa^{-1} u + \nabla p &= 0 \quad \quad \quad \text{in}  \quad \Omega\\
\text{div} (u) &= f \quad \quad \quad \text{in}  \quad \Omega\\
u\cdot n &= u_N \quad \quad \text{on}  \quad \Gamma_N \\
p &= p_D \quad \quad \text{on}  \quad \Gamma_D
\end{align*}
where $\kappa$ represents the heterogeneous permeability field which can be generated using Karhunen-Loeve expansion. $f = 1$ is a constant source term, $\Omega$ is a squared computational domain $[0,1]\times[0,1]$, and $\Gamma_N  \cup \Gamma_D = \partial \Omega$. The boundary conditions are $u_N = 0$ at $[0,1]\times \{0\}$ and  $[0,1]\times \{1\}$, $p_D = 1$ at $ \{0\}\times[0,1]$, and $p_D = 0$ at $ \{1\}\times[0,1]$.

Specifically, the permeability fields $\kappa(x; \mu)$ can be constructed as follows:
\[
\kappa^H(x; \mu) = \kappa_0 +\displaystyle{\sum_{j=1}^{p}} \mu_j  \sqrt{\xi_j} \Phi_j(x)
\] 
where $ \kappa_0$ is a constant permeability, denotes the mean of the random field. $\displaystyle{\sum_{j=1}^{p} \mu_j \sqrt{\xi_j} \Phi_j(x)}$ corresponds to a random contribution obtained from Karhunen-Loeve expansion, and describes the uncertainty in the permeability field. $\mu_j$ are random numbers drawn from i.i.d $N(0,1)$. $(\sqrt{\xi_j}, \Phi_j(x))$ are the eigen-pairs obtained from a Gaussian covariance kernel:
\begin{equation*}
\text{Cov} (x_i, y_i; x_j, y_j) = \sigma \exp(\frac{|x_i -x_j|^2}{l_x^2} -\frac{|y_i -y_j|^2}{l_y^2}  )
\end{equation*}
where we choose $[l_x, l_y]=[0.2, 0.3]$, $\sigma = 2$ and $p=32, 64, 128$ in our example. 

In the discretized system, we use $\text{RT}_0$ element for the velocity space $V_h$, and piecewise constant element $P_0$ for the pressure solution space $Q_h$. 
\begin{align*}
a(u,v)+b(v,p) &= \int_{\Gamma_\Omega} p_D v \cdot n  \text{d} s   & \text{ for all }  v\in V_h\\
b(u,q) &= -(f,q)   &\text{ for all }  q\in Q_h
\end{align*}
where $a(u,v) = \int_{\Omega} \kappa^{-1} u \cdot v$, and $b(v,p) = -\int_\Omega p \; \text{div} v $.

The discrete system has the following matrix formulation
\begin{equation}\label{eq:vel_mat_single}
\begin{bmatrix}
A_h(\kappa) & B_h^T   \\
B_h & 0
\end{bmatrix}
\begin{bmatrix}
u_h \\
p_h
\end{bmatrix} =
\begin{bmatrix}
G_D  \\
-F
\end{bmatrix}
\end{equation}

However, due to the multiscale nature of $\kappa$, a sufficiently fine mesh is required to resolve all scale properties. Thus the fine matrix $\begin{bmatrix}
A_h(\kappa) & B_h^T   \\
B_h & 0
\end{bmatrix}$ has a large size, leading to some difficulties in solving the linear system. To overcome these, one can develop a reduced order model as a surrogate. Numerous mixed multiscale methods have been explored \cite{MixedGMsFEM, Arbogast_mixed_MS_11, ae07}. For example, in  \cite{MixedGMsFEM}, one aims to construct velocity multiscale basis in each local coarse region, and use the piecewise constant on coarse grid to approximate the pressure. Typically, let $N_u^H$ be the dimension of the multiscale velocity space, and denote by $R_u$ the matrix assembled using multiscale velocity basis in every row, then $R_u$ maps from $\mathbb{R}^{N_u^h}$ to $\mathbb{R}^{N_u^H}$, where $N_u^h$ is the fine degrees of freedom for the velocity. Similarly, denote by $R_p$ the matrix containing coarse grid piecewise constant basis for pressure which maps from $\mathbb{R}^{N_p^h}$ to $\mathbb{R}^{N_p^H}$. Then one can rewrite the system \ref{eq:vel_mat_single} in the following form
\begin{equation}\label{eq:vel_mat_single_coarse}
\begin{bmatrix}
A_H & B_H^T   \\
B_H & 0
\end{bmatrix}
\begin{bmatrix}
u_H \\
p_H
\end{bmatrix} =
\begin{bmatrix}
R_u & 0   \\
0 & R_p
\end{bmatrix}
\begin{bmatrix}
A_h(\kappa) & B_h^T   \\
B_h & 0
\end{bmatrix}
\begin{bmatrix}
R_u^T & 0   \\
0 & R_p^T
\end{bmatrix}
\begin{bmatrix}
u_H \\
p_H
\end{bmatrix}=
\begin{bmatrix}
0   \\
-F_H
\end{bmatrix}
\end{equation}
where $\begin{bmatrix}
R_u & 0   \\
0 & R_p
\end{bmatrix}$ can be viewed as an encoder which maps from fine grid to coarse grid (upscaling), and $\begin{bmatrix}
R_u^T & 0   \\
0 & R_p^T
\end{bmatrix}$ acts as an decoder which maps from coarse grid to fine grid (downscaling).

The coarse grid solver reveals its efficiency when we need to solve flow problem with varying source or boundary conditions, while with a fixed permeability field. However, in a lot of applications, it is more interesting to solve for the velocity $u$ given different $\kappa$. When the permeability fields vary, one needs to reconstruct the multiscale basis (reconstruct the matrix $R_u$) in the above mentioned multiscale method framework, which is not practical. 

As discussed in \cite{wang_multiphase}, we will construct an encoding-decoding type of network to approximate the relationship between the permeability fields $\kappa$ and fine grid velocity solution $u_h$. That is,  $u = \mathcal{N}(\kappa; \theta)$. The proposed network structure is in analogy to the coarse-scale solver but will take permeability fields as input without constructing a set of the multiscale bases for each case. 

The idea is to first apply a few convolution layers to extract features from the input permeability with size $\sqrt{N_p^h} \times \sqrt{N_p^h}$, and then project the extracted features on a coarser mesh by employing an average pooling layer. The intermediate output is then flattened and is linked to $N_p^H$ neurons with a fully connected layer. This procedure is in analogy to upscaling. We will then reshape the hidden coarse grid features to an image with size $\sqrt{N_p^H} \times \sqrt{ N_p^H}$. A few locally connected layers or convolution layers are followed to mimic the coarse grid solver. 

After that, the resulting hidden features are flattened again and are fully connected with $N_u^H$ neurons in the next layer, where $N_u^H$ is the degrees of freedom for multiscale velocity space. It is natural to represent the coarse grid velocity using a vector since the degrees of freedom are not located at coarse grid centers, which makes it not obvious to reshape it as a square image. Finally, we decode the coarse level features using a densely connected layer, and we obtain fine grid velocity output with dimension $N_u^h$. The network architecture is illustrated in Figure \ref{fig:vel_lcn}.
 
 However, the last downscaling layer is still fully connected. Due to the large degrees of freedom for the velocity solution, the last fully connected layer contributes very large numbers of trainable weights. Here, we would like to use our proposed sparse learning method to tackle this difficulty.

\begin{figure}[!hbt]
	\centering
	\includegraphics[scale=0.35]{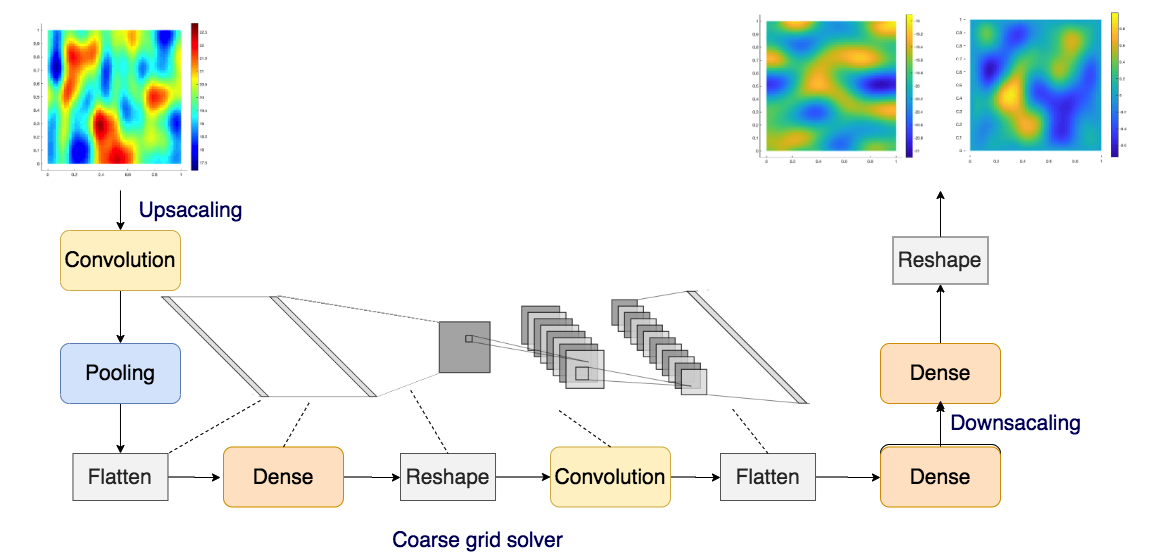}
	\caption{An illustration of the network architecture for flow approximation.}
	\label{fig:vel_lcn}
\end{figure}

The training and testing data can be generated by solving the equations with a mixed finite element method on the fine grid for various permeability fields. An illustrations of the permeability fields for $p=32, 64, 128$ and corresponding their corresponding solutions are presented in \ref{fig:kle_sol}. We can see that when $p$ becomes larger, the velocity solutions exhibit many more scale features. 

\begin{figure}[!hbt]
	\begin{subfigure}[t]{1.0\textwidth}
		\centering
		\includegraphics[width=.24\textwidth]{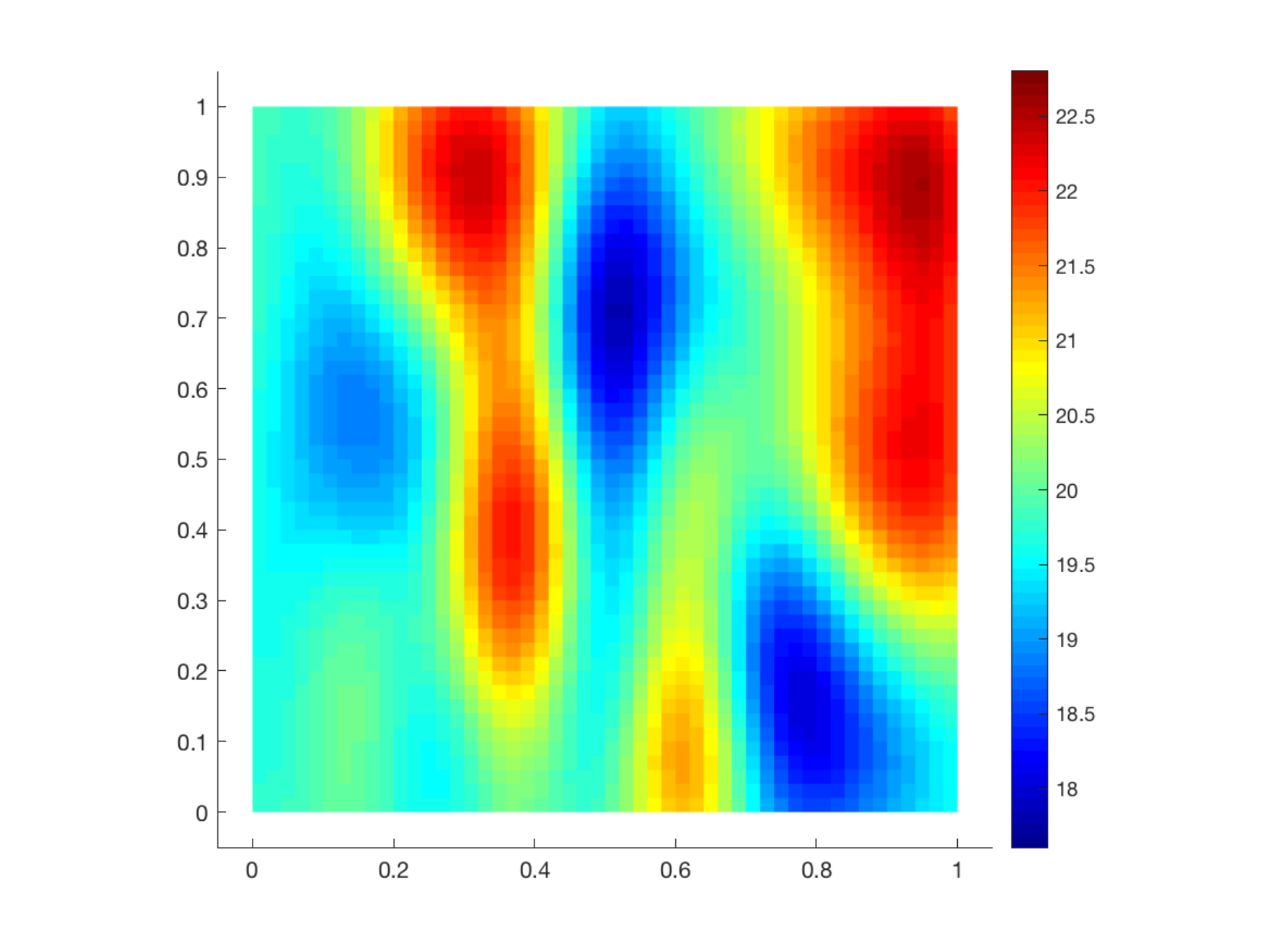} \hfill
		\includegraphics[width=.24\textwidth]{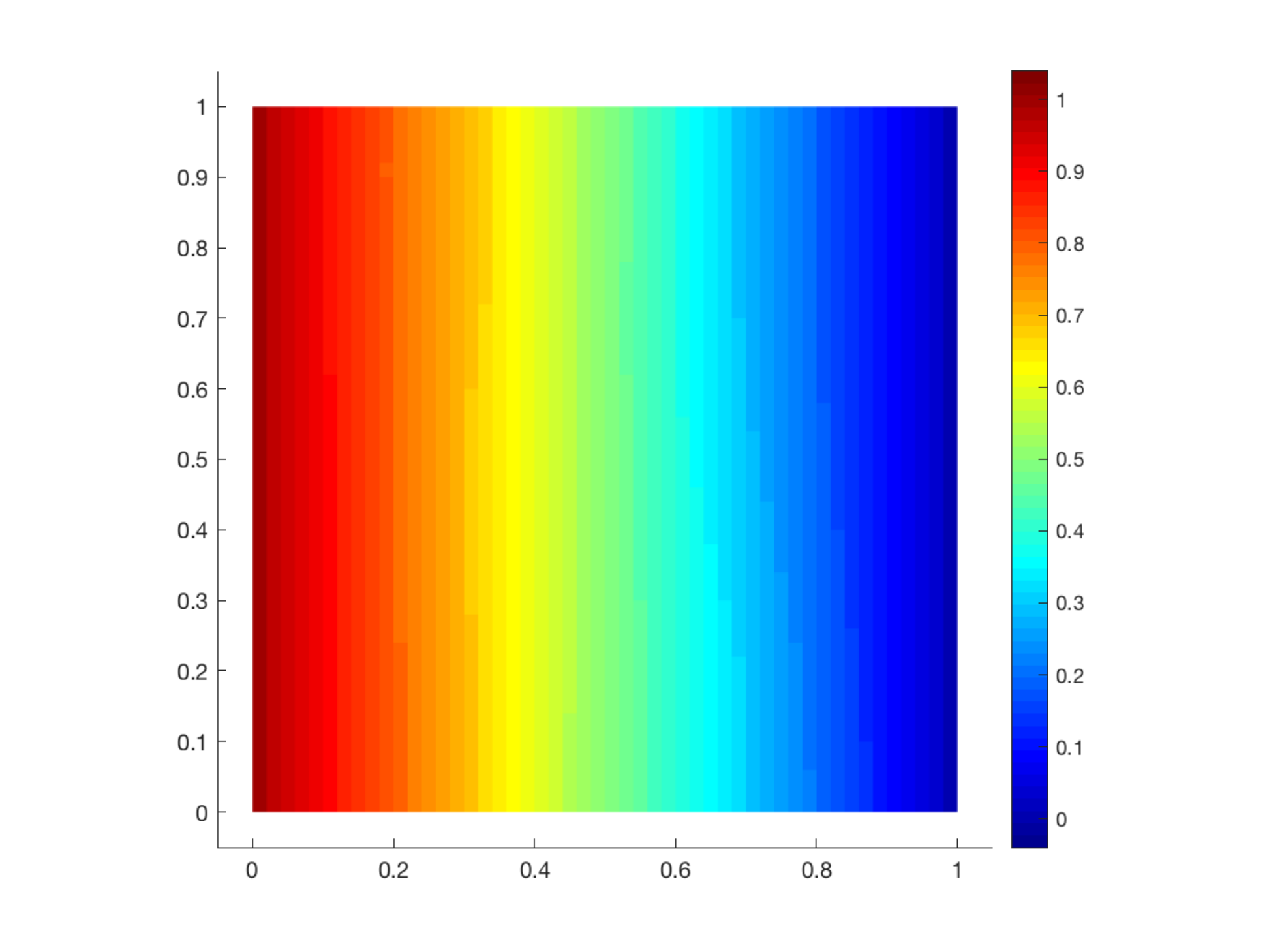}\hfill
		\includegraphics[width=.24\textwidth]{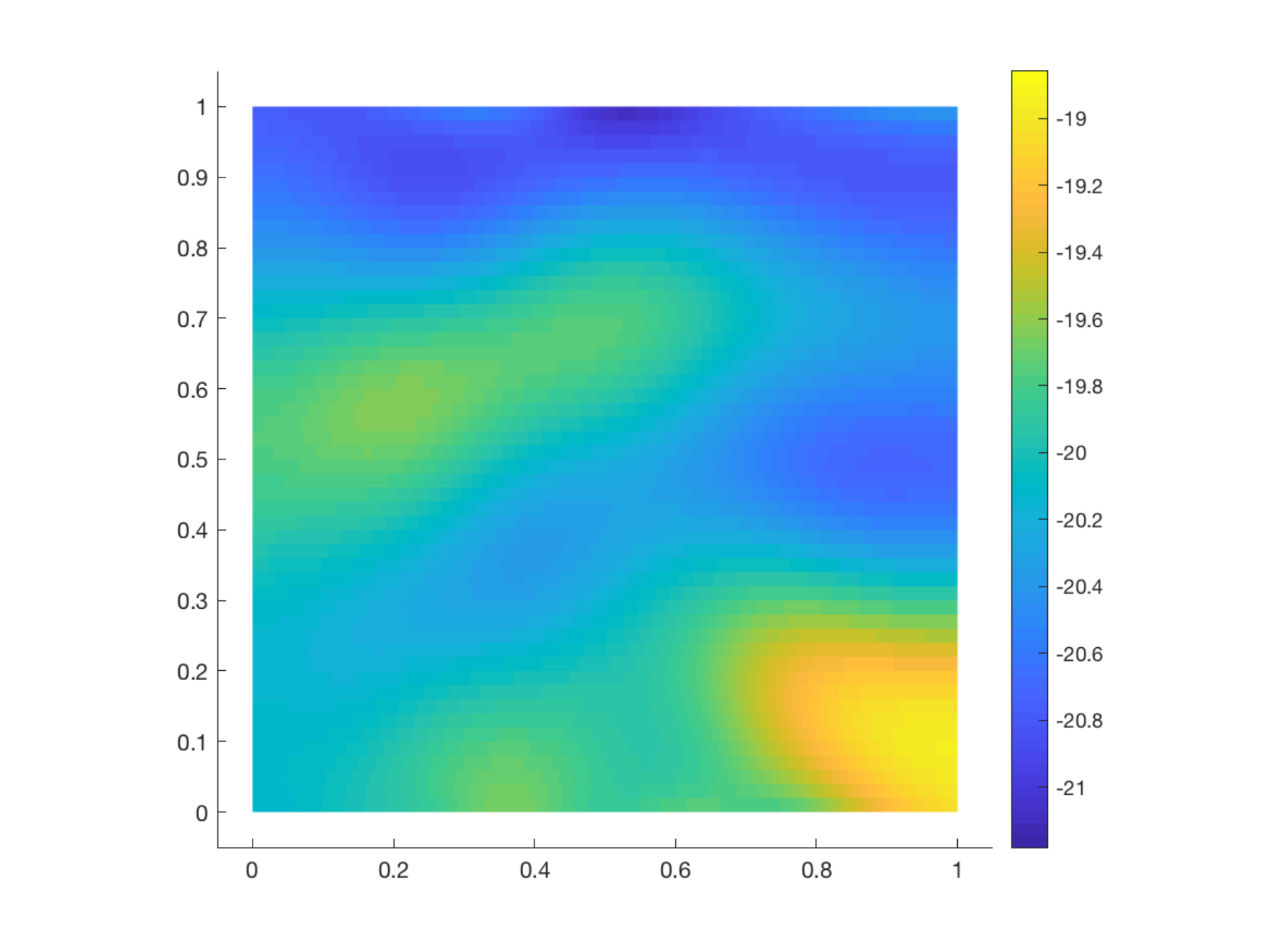}	\hfill	
		\includegraphics[width=.24\textwidth]{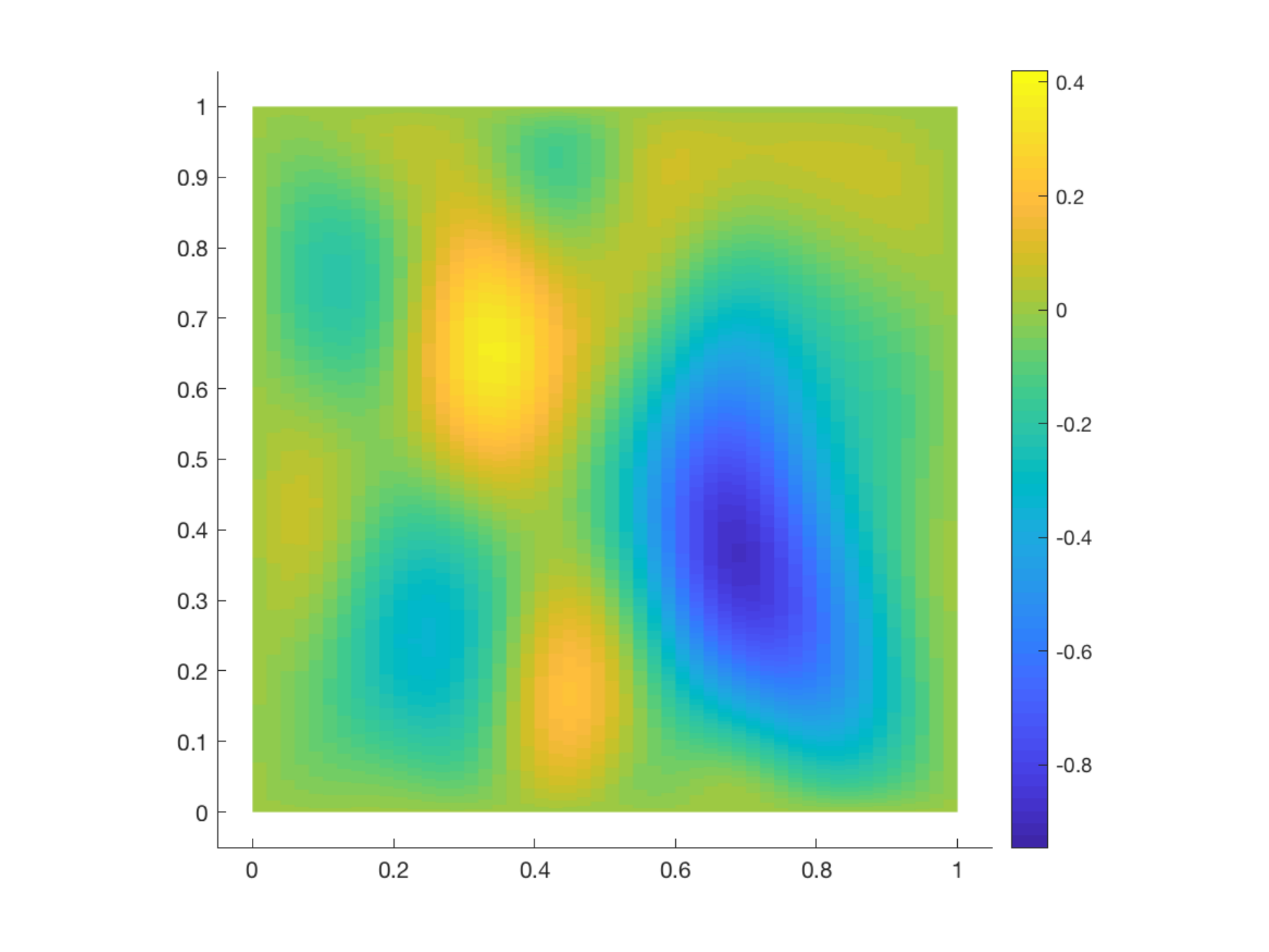}\hfill
		\caption{KLE 32}
	\end{subfigure}
			\begin{subfigure}[t]{1.0\textwidth}
		\centering
		\includegraphics[width=.24\textwidth]{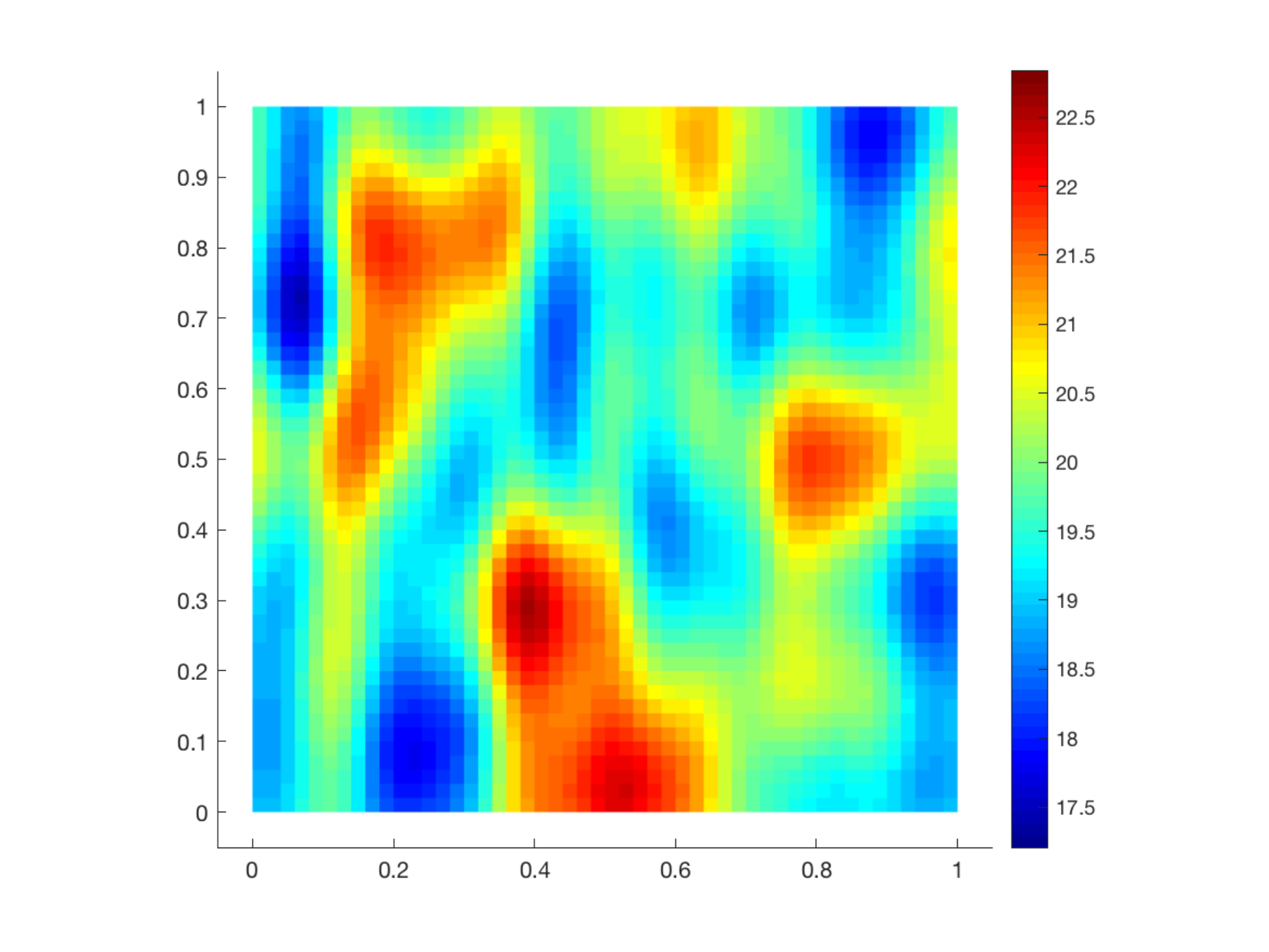}\hfill
		\includegraphics[width=.24\textwidth]{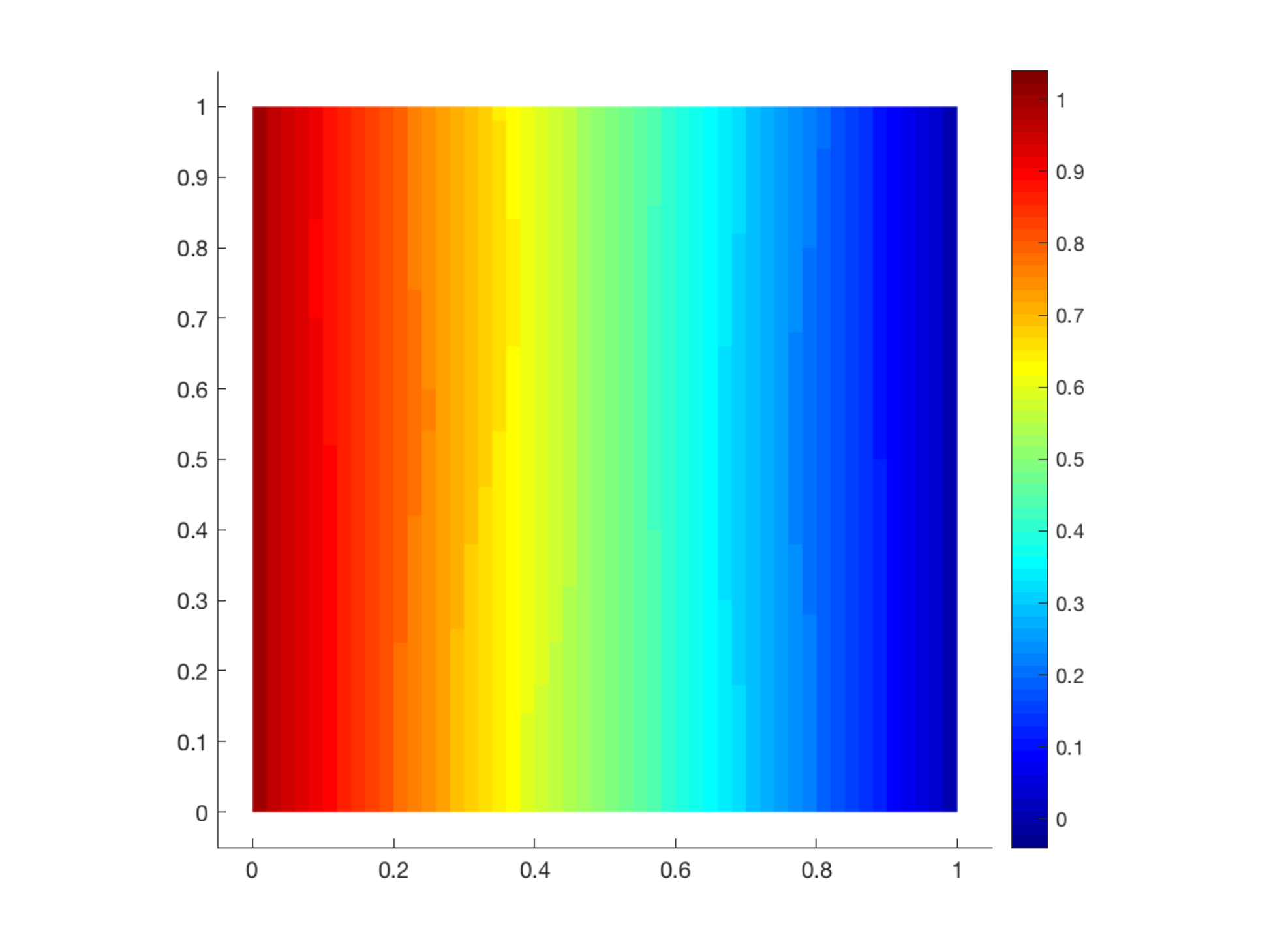}\hfill
		\includegraphics[width=.24\textwidth]{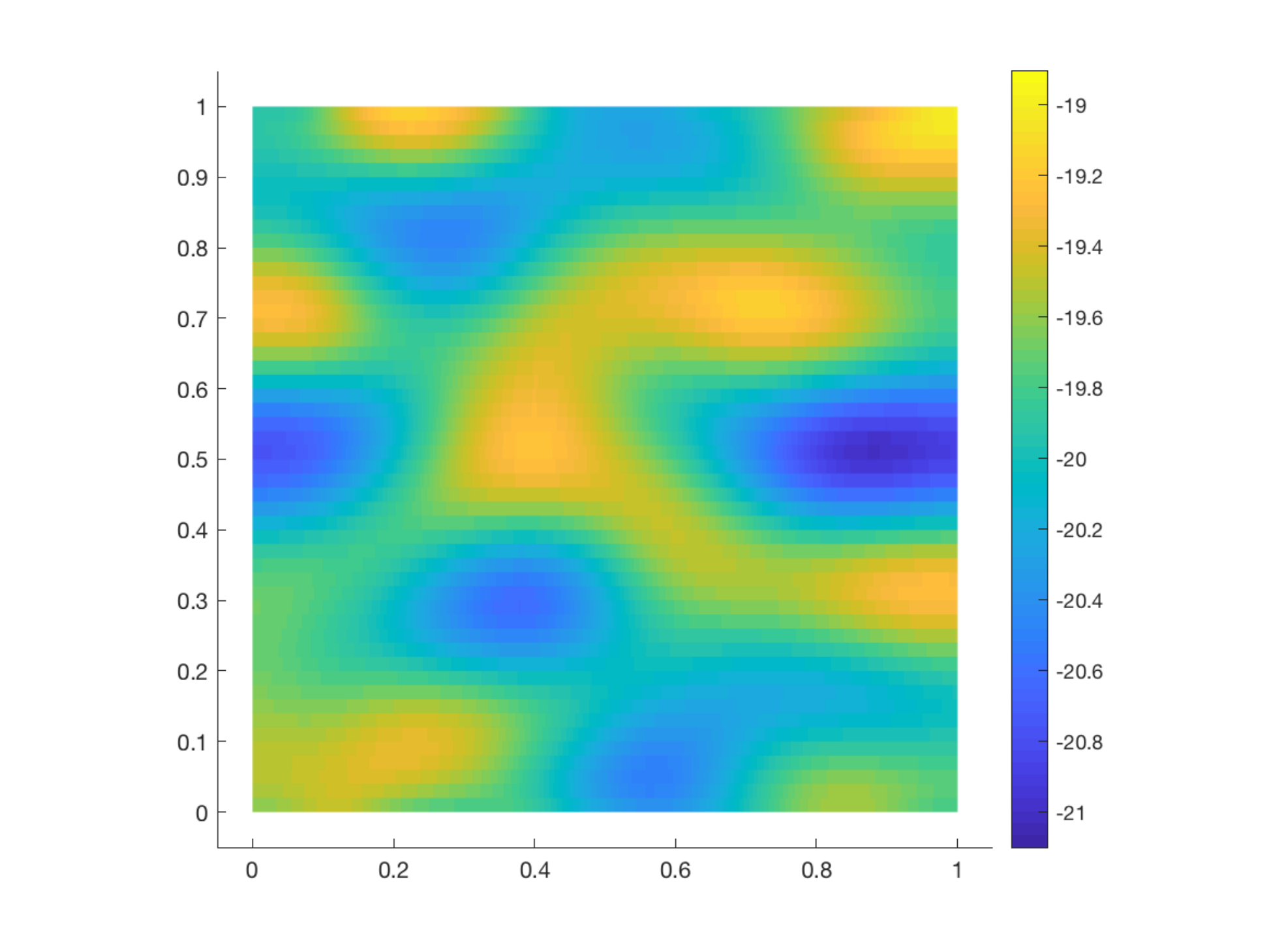}	\hfill	
		\includegraphics[width=.24\textwidth]{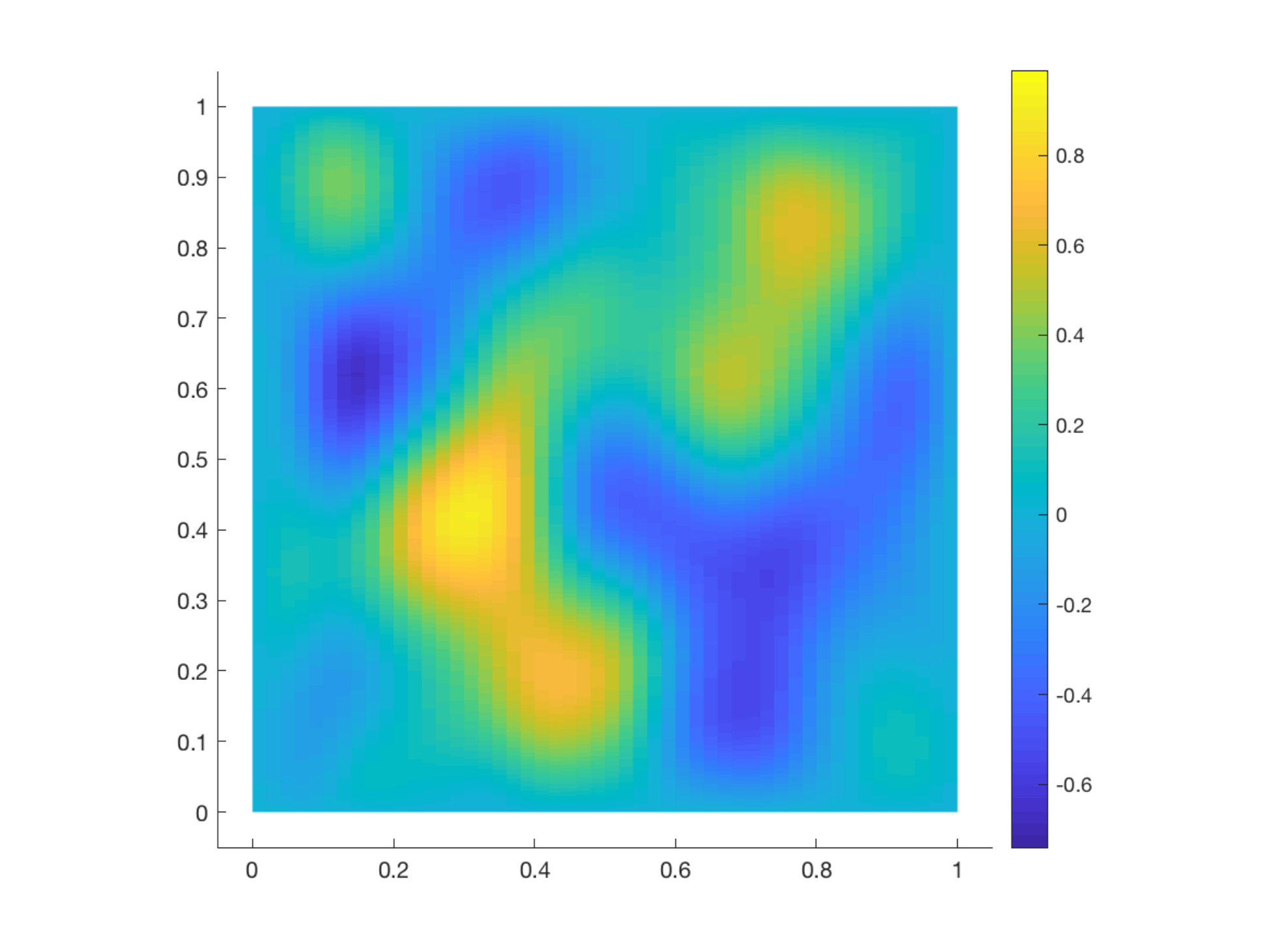}\hfill
		\caption{KLE 64}
	\end{subfigure}
		\begin{subfigure}[t]{1.0\textwidth}
	\centering
	\includegraphics[width=.24\textwidth]{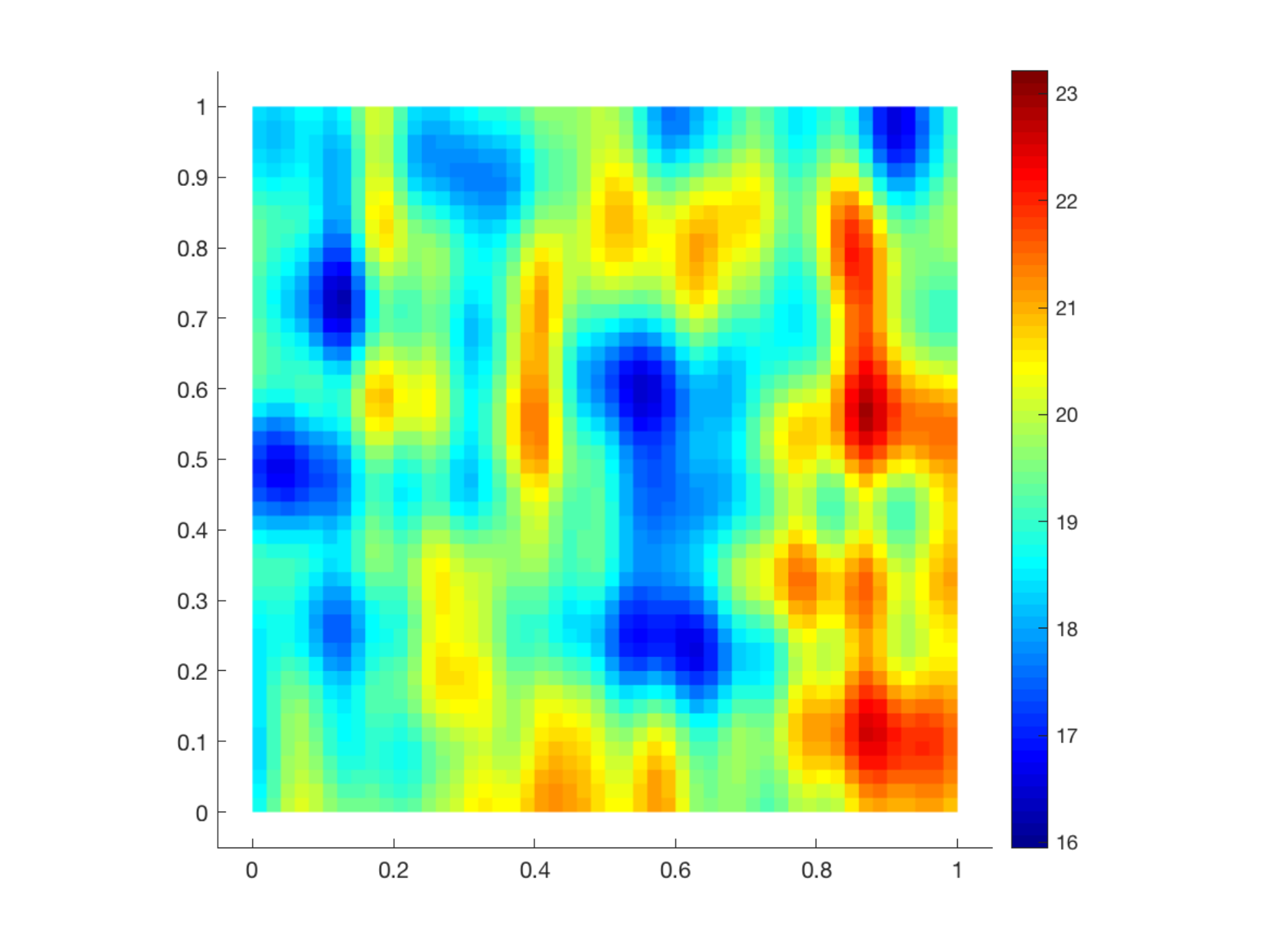}\hfill
	\includegraphics[width=.24\textwidth]{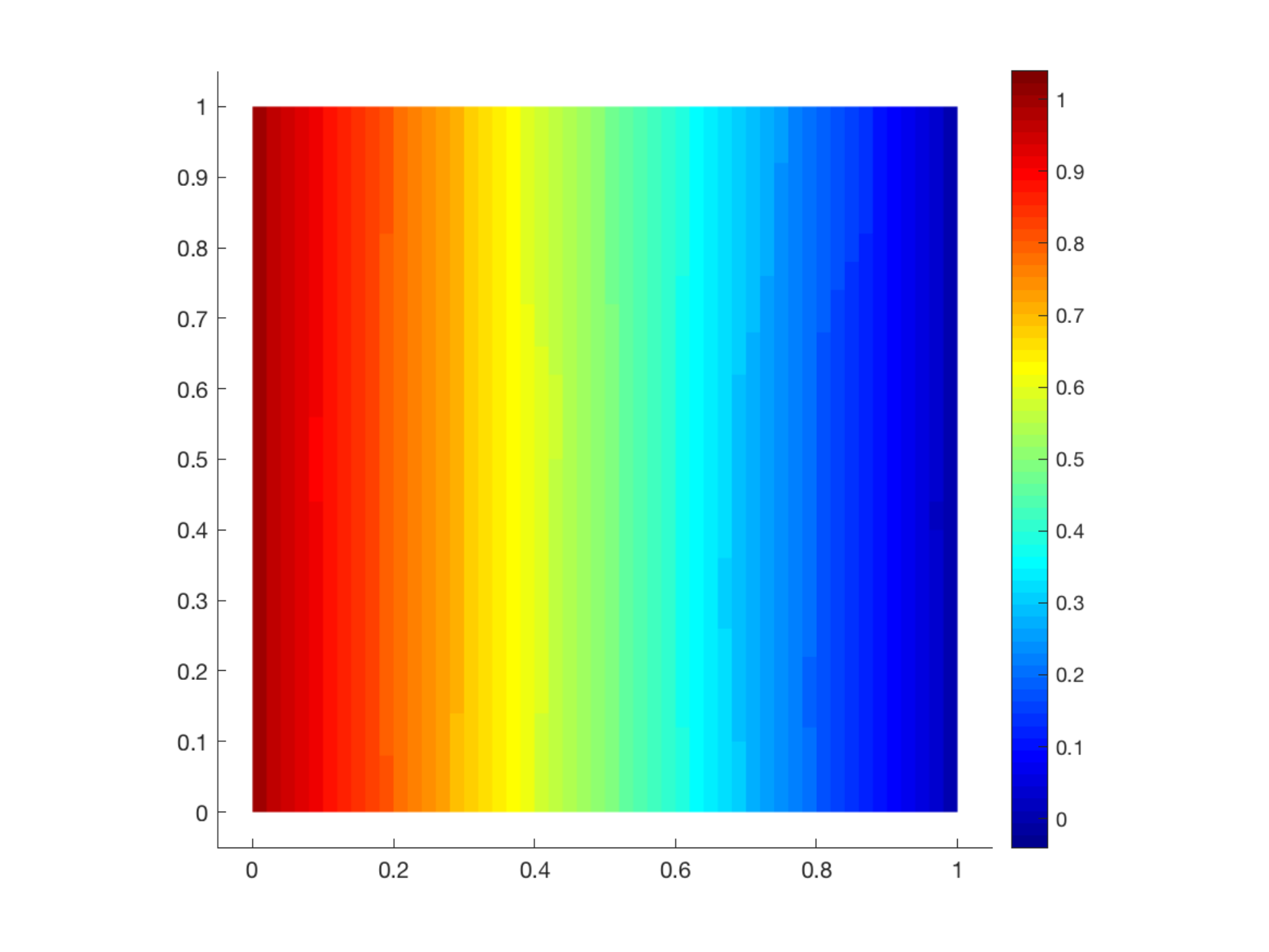}\hfill
	\includegraphics[width=.24\textwidth]{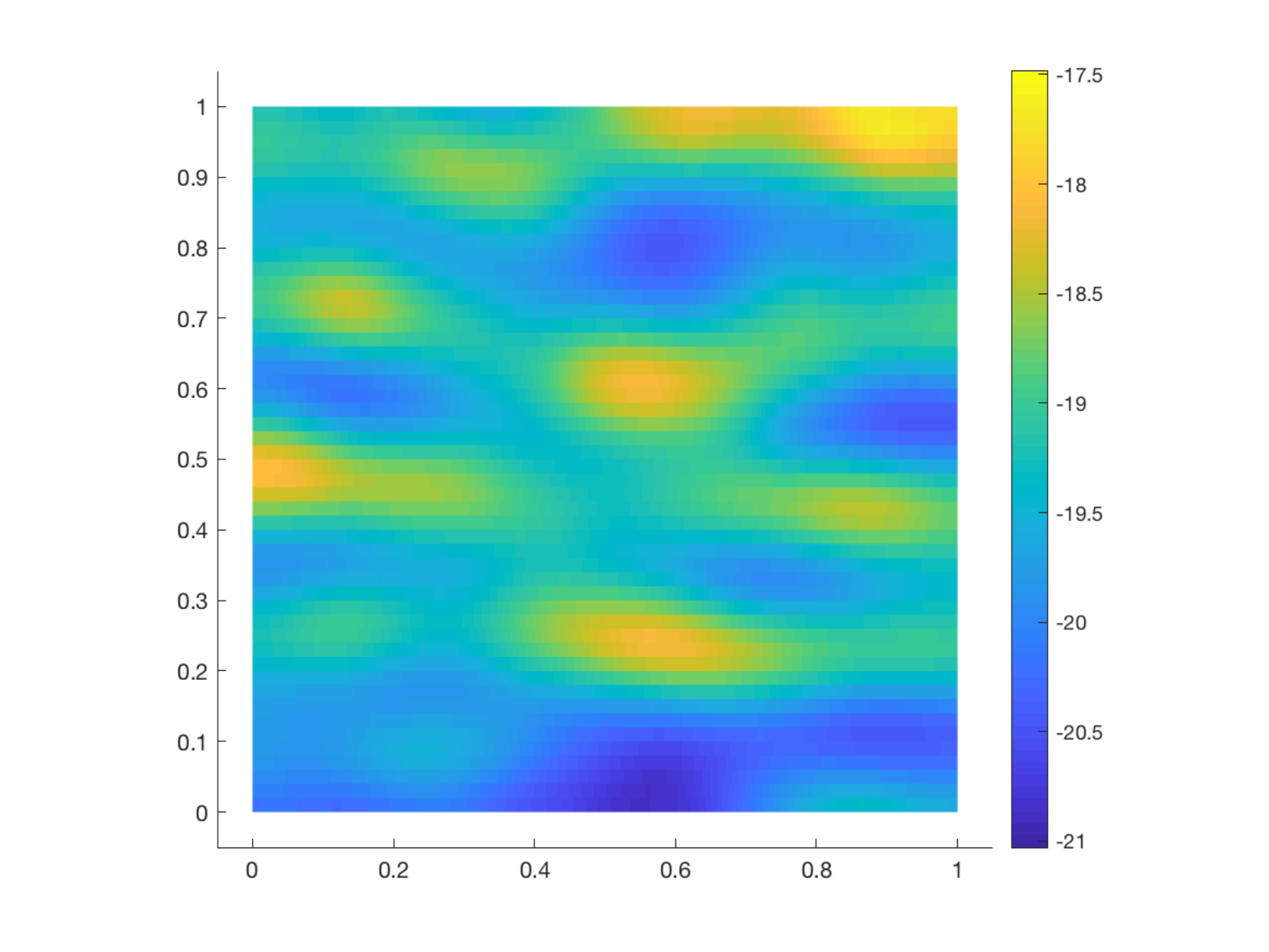}	\hfill	
	\includegraphics[width=.24\textwidth]{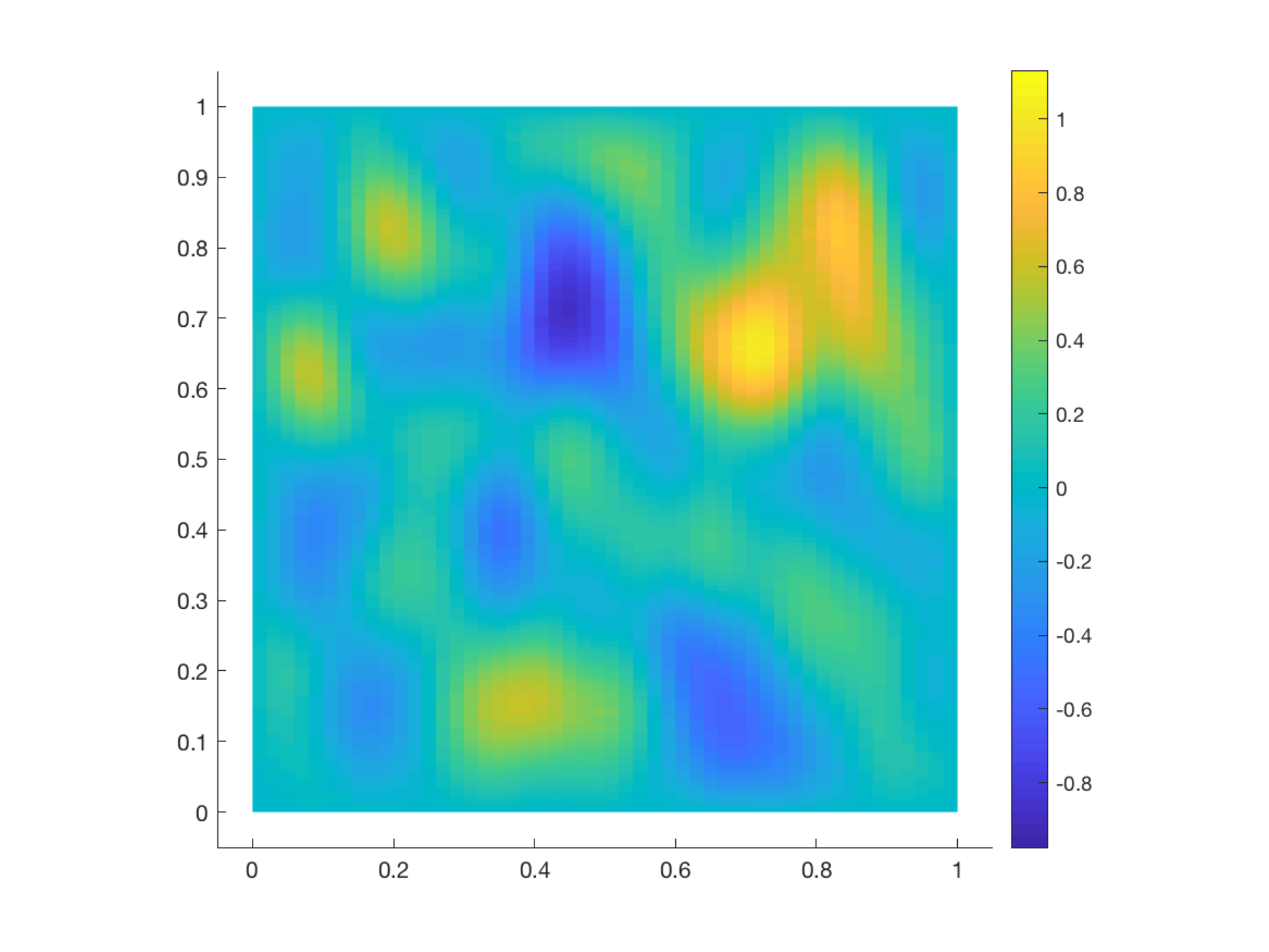}\hfill
	\caption{KLE 128}
\end{subfigure}
	\caption{ Illustrations of the permeability fields when using different number of terms in KLE expansion and corresponding solutions. In each subplot, permeability (upper left), pressure solution (upper right), horizontal velocity magnitude (lower left) and vertical velocity magnitude (lower right).}\label{fig:kle_sol}
\end{figure}

We generate $1500$ samples pairs $(\kappa_i, u_h^i)$, and randomly pick $1300$ samples to train the network, and take the rest of the samples for validation. The size of an input permeability is $50\times 50$, an output velocity solution vector is $5,100$. The network first uses 2 convolution layers with window size $3\times 3$, and $64$ and $32$ channels respectively. Then, an average pooling layer with pool size $2\times 2$ is followed by a flatten layer and then a fully connected layer with $100$ nodes. This part of the network can be viewed as an encoder. Then, a reshaping layer, another two convolution layers, a flatten layer, and a dense layer with $800$ neurons are used to mimic the coarse grid solver. In the end, a dense layer is used as a decoder. The total number of parameters is $8,252,320$ in the entire network, and the layers we choose to perform sparse learning contains $6,110,624$ parameters.

The numerical results using SGLD, PSGLD, SGLD-SA and PSGLD-SA are presented in Table \ref{tab:example2}. Denote by $e_1 = \frac{\norm{u_{\text{pred}} - u_{\text{true}} }_{L^2}}{\norm{u_{\text{true}} }_{L^2}}$ and $e_2 = \frac{\norm{u_{\text{pred}} - u_{\text{true}} }_{L^2_{\kappa}}}{\norm{u_{\text{true}} }_{L^2_{\kappa}}}$ where $||u||_{L^2_{\kappa}} = \int_{\Omega} \kappa^{-1} |u|^2$. The mean relative errors $e_!$ and $e_2$ among 300 testing samples are shown. We can see that, the results using PSGLD-SA outperforms SGLD-SA in all three cases when $p=32, 64, 128$. Some predicted results for different values of $p$ in KLE are presented in Figures \ref{fig:kle32}, \ref{fig:kle64} and \ref{fig:kle128}. We can actually see that, the predictions using SGLD-SA is very poor, but the results using PSGLD-SA are very similar to the ground truth.

In this example, we choose the sparse rate to be $50\%$ and $70\%$. By choosing appropriate hyperparameters, we can achieve similar accuracy for the dense cases and sparse cases as shown in Table \ref{tab:example2}. This indicates enforcing sparsity using our method can maintain accuracy while reducing storage/computational cost.
But if the sparse rate is too large, we find it's hard to get comparable results since over sparse network may not be sufficient to represent the properties of the target map of interest. On the other hand, compare PSGLD and SGLD, we notice that applying preconditioners can provide better results. The learning curves are presented in \ref{fig:kle_learning_curves}. It shows that PSGLD-SA converges faster than SGLD-SA or vanilla PSGLD in all three cases.

\begin{table}[!htb]
	\centering
	\begin{tabular}{|c  |c  | c  | }
		\hline
			\multicolumn{3}{|c|}{Dense} \\ \hline
			&	PSGLD ($e_1$/ $e_2$ \% )& 	SGLD($e_1$/ $e_2$ \% )  \\  \hline
	 KLE32 &0.75/0.57  & 2.37  /2.17  \\  \hline
	KLE64 & 0.82/0.63 &2.38 /2.25    \\  \hline
	KLE128  &2.13 /1.93  &2.90 /2.60       \\  \hline
    		&	PSGLD-SA ($e_1$/ $e_2$ \% )& 	SGLD-SA ($e_1$/ $e_2$ \% )  \\  \hline
    			\multicolumn{3}{|c|}{Sparse rate 50\%} \\ \hline
    KLE32 &0.59/0.56  & 2.67  /2.35    \\  \hline
	KLE64 & 0.78 /0.58 &2.68 /2.41      \\  \hline
	KLE128  &1.60 /1.31  &3.47 /3.00       \\  \hline
			\multicolumn{3}{|c|}{Sparse rate 70\%} \\ \hline
			    KLE32 &0.58/ 0.51  &2.28 /2.10    \\  \hline
			KLE64 & 0.76 /0.61 &2.40 /2.97      \\  \hline
			KLE128  &1.79 /1.60  & 3.51/3.02  \\  \hline
	\end{tabular}
	\caption{Mean errors between the true and predicted velocity solutions using SGLD, PSGLD, SGLD-SA, and proposed PSGLD-SA. Mean errors of $300$ testing cases.}\label{tab:example2}
\end{table}

\begin{figure}[!ht]
	\begin{subfigure}[t]{0.3\textwidth}
		\centering
		\includegraphics[scale=0.3]{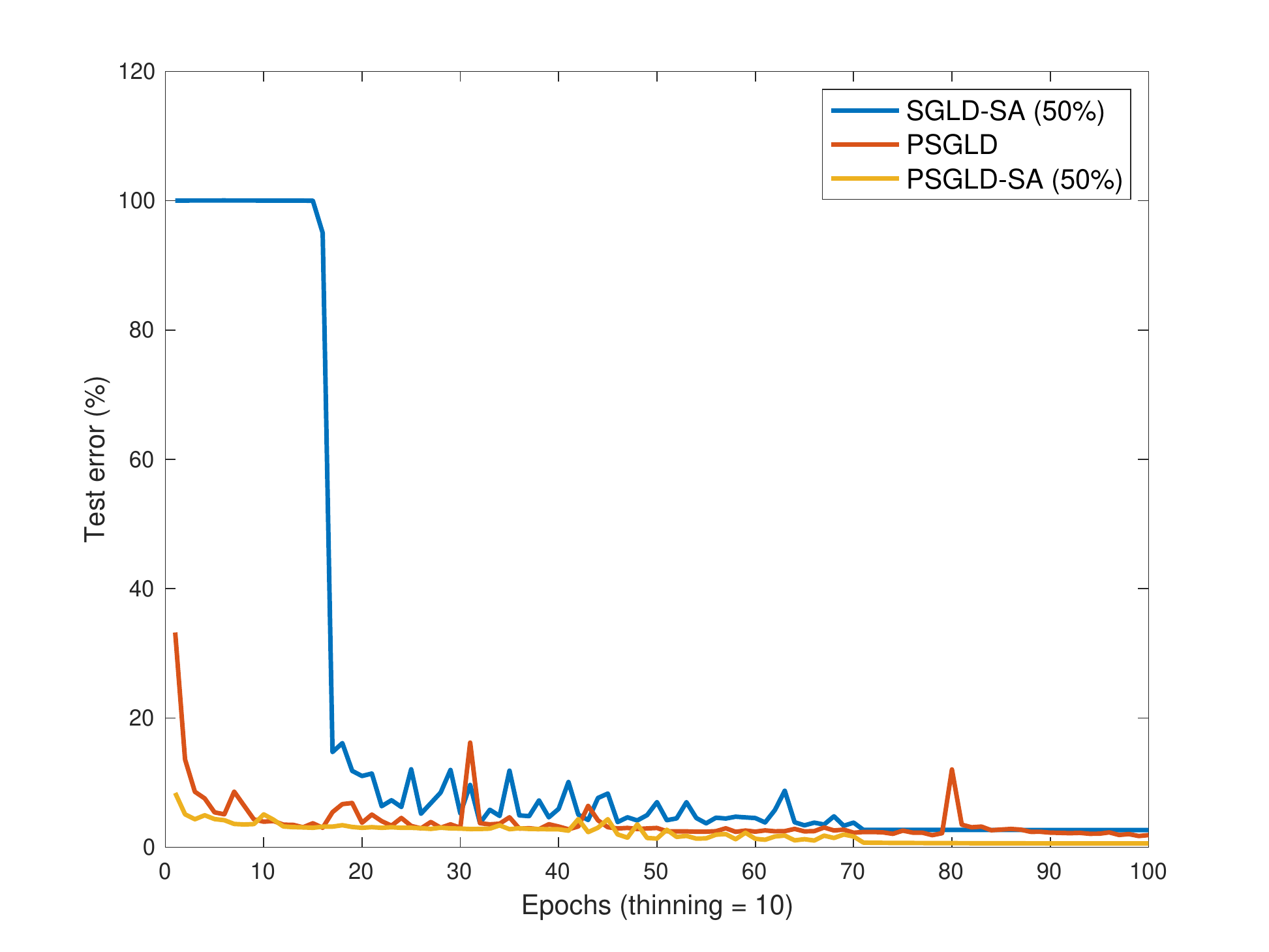}
		\caption{KLE 32, learning curves}
	\end{subfigure}
	\begin{subfigure}[t]{0.3\textwidth}
		\centering
		\includegraphics[scale=0.3]{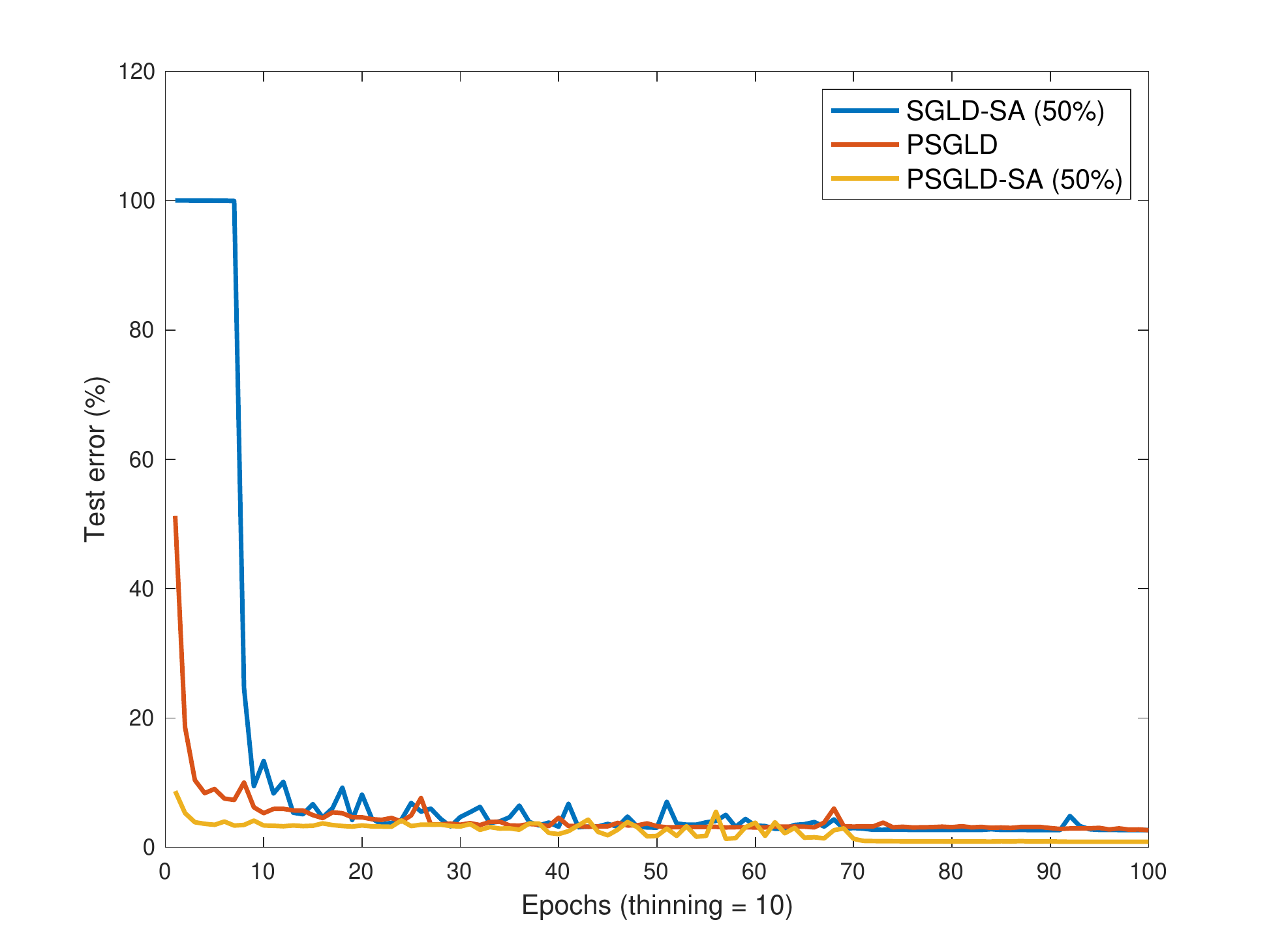}	
		\caption{KLE 64,  learning curves}
	\end{subfigure}
	\begin{subfigure}[t]{0.3\textwidth}
		\centering
		\includegraphics[scale=0.3]{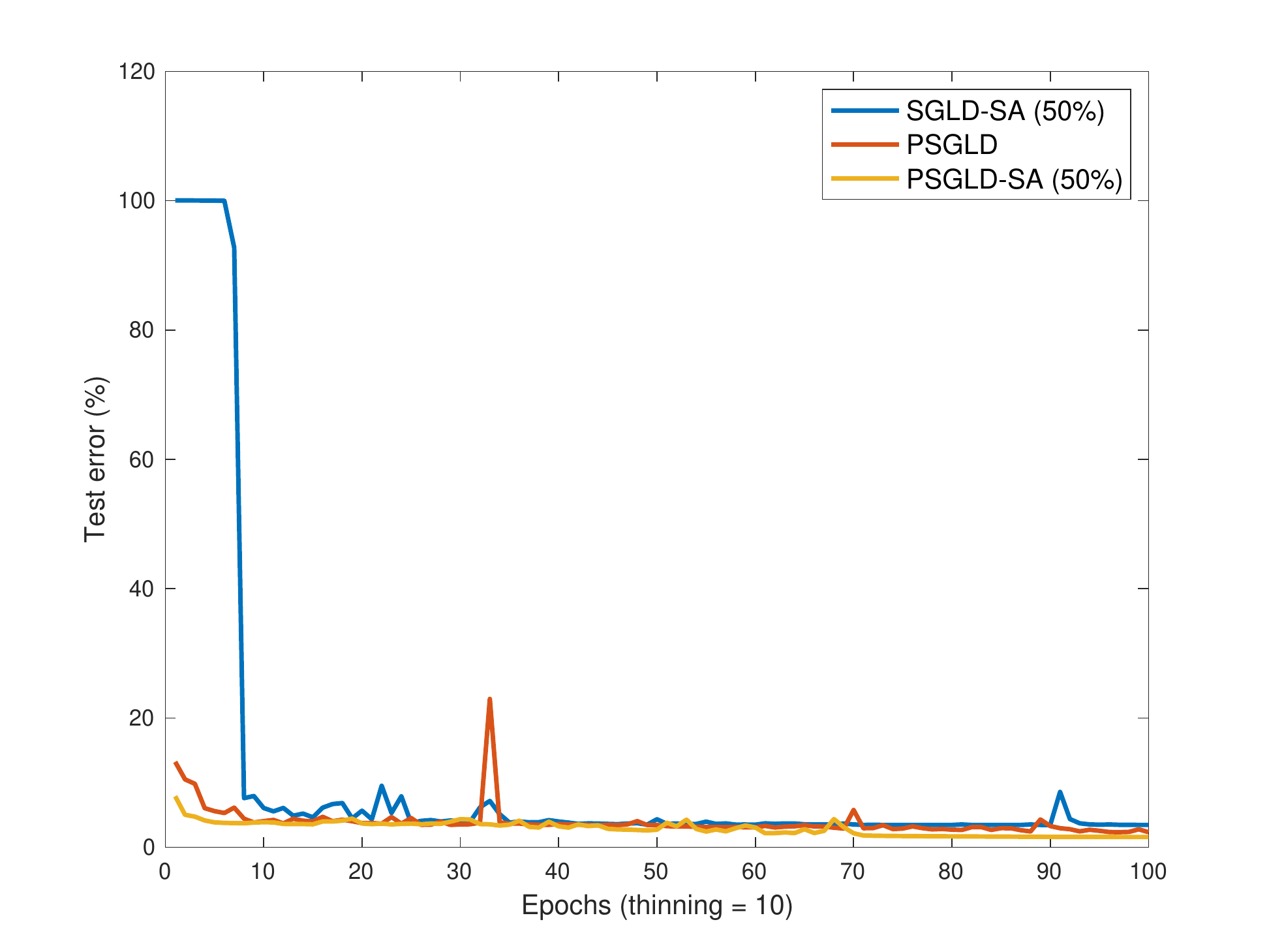}	
		\caption{KLE 128,  learning curves}
	\end{subfigure}
	\caption{Learning curves. In each sub-figure, there are comparison among test errors SGLD with sparse approximation, vanilla PSGLD, and PSGLD with sparse approximation.}\label{fig:kle_learning_curves}
\end{figure}

\begin{figure}[!hbt]
	\begin{subfigure}[t]{0.9\textwidth}
		\centering
		\includegraphics[scale=0.7]{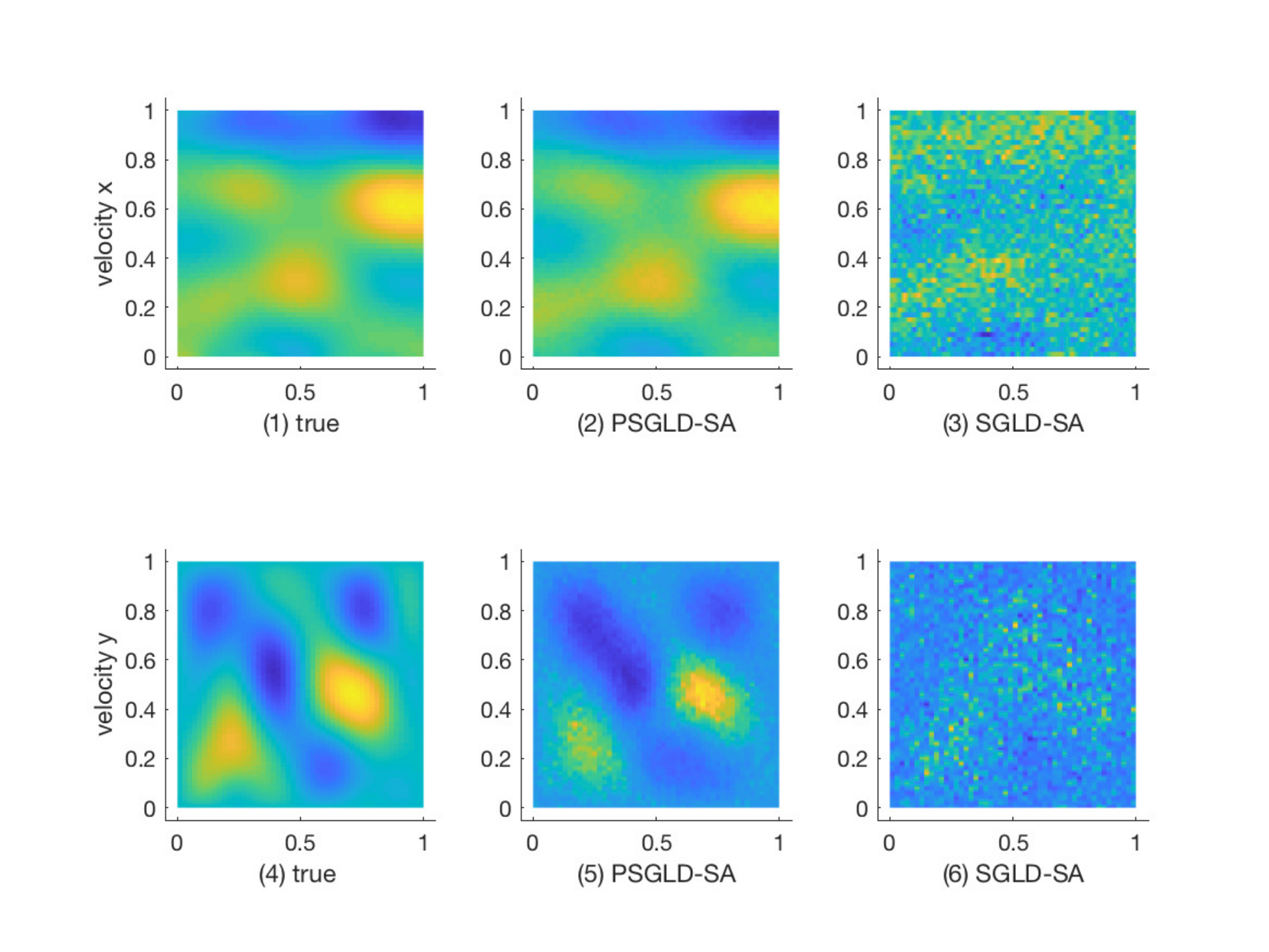}
		\caption{KLE 32, test sample 1}
	\end{subfigure}

	\begin{subfigure}[t]{0.9\textwidth}
		\centering
		\includegraphics[scale=0.7]{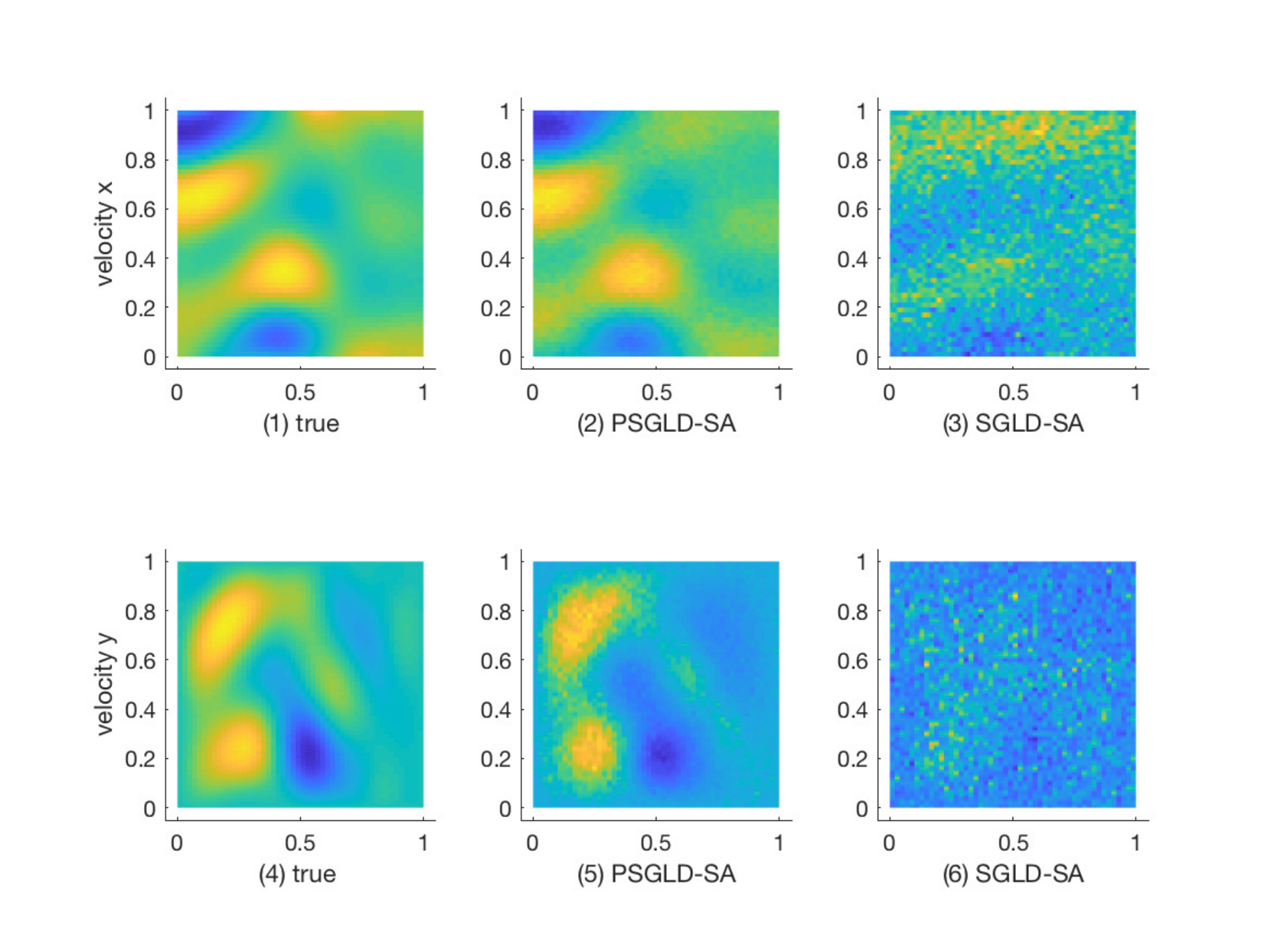}	
		\caption{KLE 32, test sample 2}
	\end{subfigure}

	\caption{KLE 32. True and prediction solutions. In each sub-figure, the first row represents horizontal velocity solution magnitude, the second row represents vertical velocity solution magnitude. From left to right: true, PSGLD-SA prediction, SGLD-SA prediction.}\label{fig:kle32}
\end{figure}

\begin{figure}[!hbt]
\begin{subfigure}[t]{0.9\textwidth}
	\centering
	\includegraphics[scale=0.7]{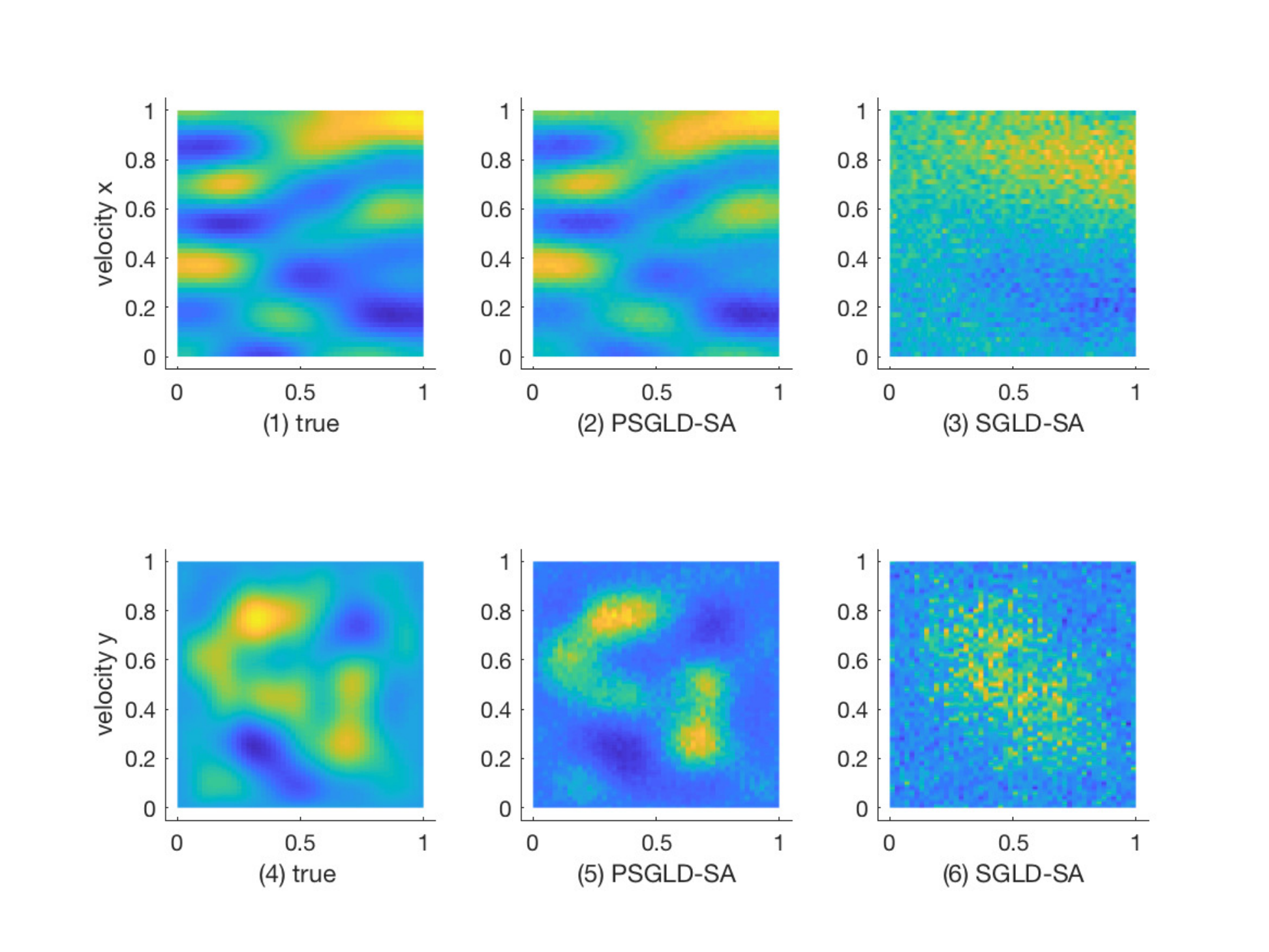}
	\caption{KLE 64, test sample 1}
\end{subfigure}%

\begin{subfigure}[t]{0.9\textwidth}
	\centering
	\includegraphics[scale=0.7]{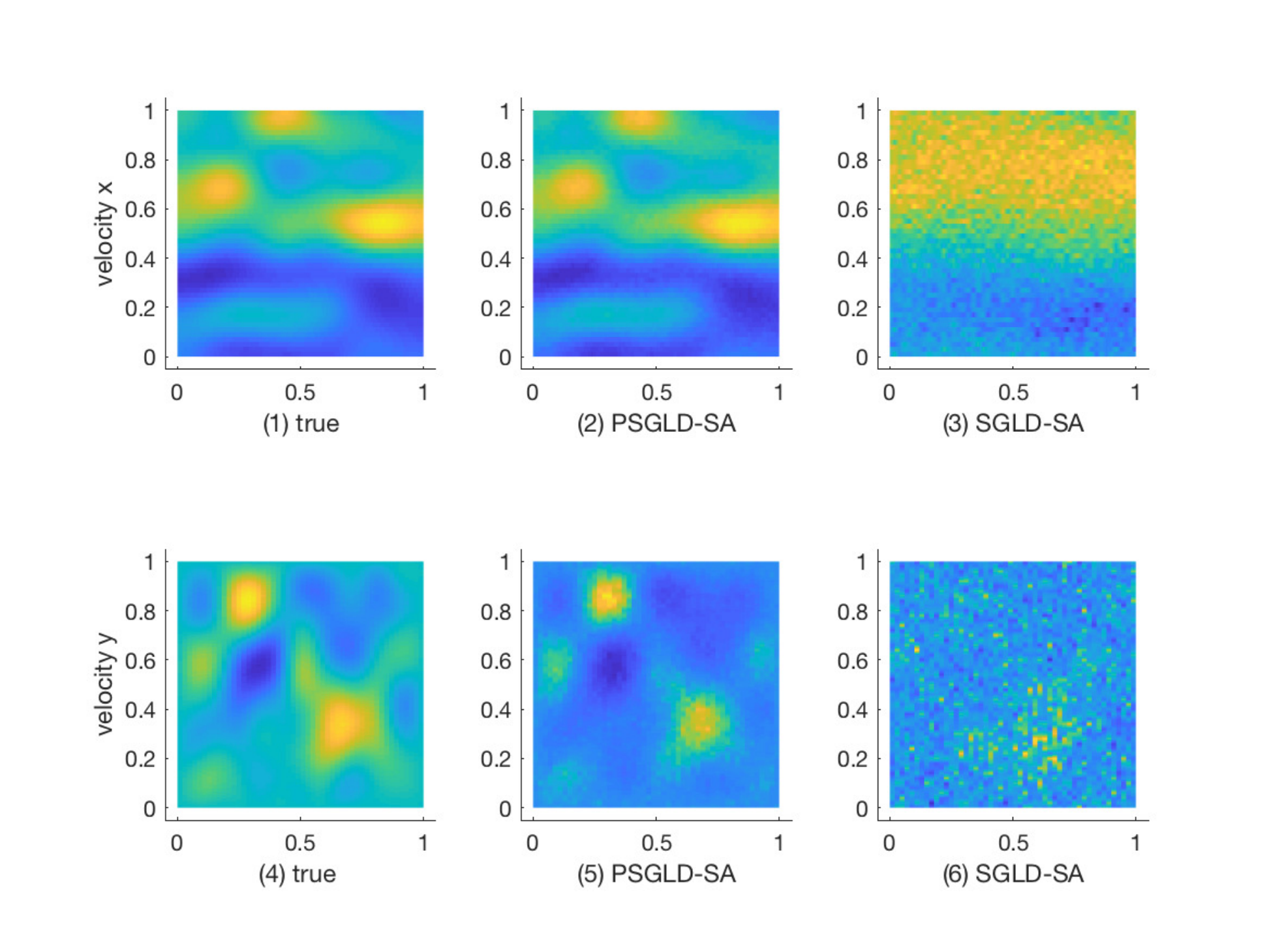}	
	\caption{KLE 64, test sample 2}
\end{subfigure}

\caption{KLE 64. True and prediction solutions. In each sub-figure, the first row represents horizontal velocity solution magnitude, the second row represents vertical velocity solution magnitude. From left to right: true, PSGLD-SA prediction, SGLD-SA prediction.}\label{fig:kle64}
\end{figure}

\begin{figure}[!hbt]
	\begin{subfigure}[t]{0.9\textwidth}
		\centering
		\includegraphics[scale=0.7]{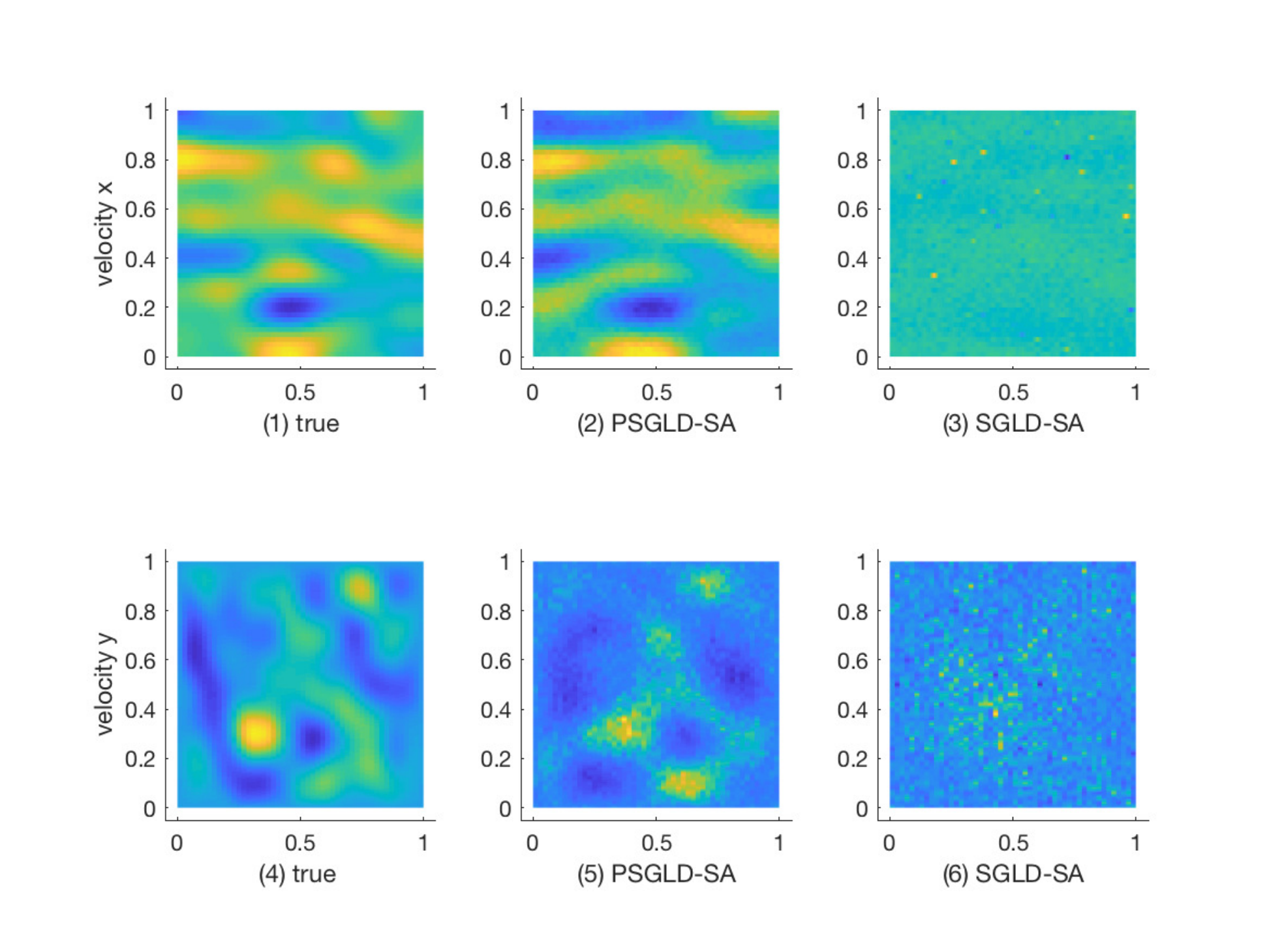}
		\caption{KLE 128, test sample 1}
	\end{subfigure}%
	
	\begin{subfigure}[t]{0.9\textwidth}
		\centering
		\includegraphics[scale=0.7]{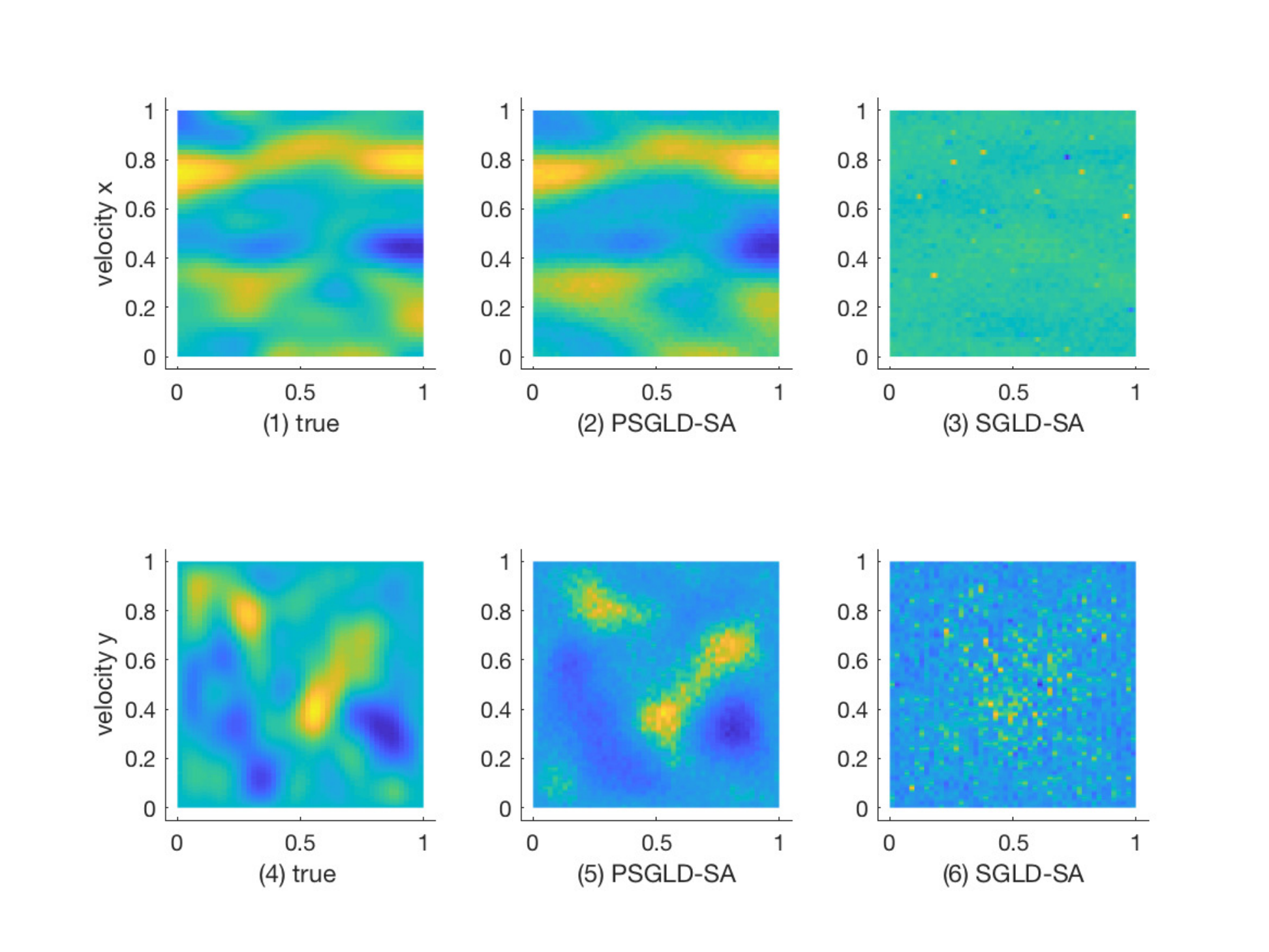}	
		\caption{KLE 128, test sample 2}
	\end{subfigure}
	\caption{KLE 128. True and prediction solutions. In each sub-figure, first row represents horizontal velocity solution magnitude, second row represents vertical velocity solution magnitude. From left to right: true, PSGLD-SA prediction, SGLD-SA prediction.} \label{fig:kle128}
\end{figure}

%
%
%
%

\subsection{Elliptic problem with channelized media}
Last, we employ the proposed algorithm to predict the solution of an elliptic problem with channelized media. The problem setup is the same as in section \ref{sec:num_ex2}. However, the background permeability fields are images of channelized media. The image size for our problem is $50\times 50$, which are patches cropped from the channelized media in \cite{laloy2018training}.
An illustration of the permeability data and corresponding solutions are presented in Figure \ref{fig:channel}. We generate $3000$ data pairs and randomly split them into $80\%$ and $20\%$ for training and testing purposes respectively.

\begin{figure}[!hbt]
	\begin{subfigure}[t]{1.0\textwidth}
		\centering
		\includegraphics[width=.24\textwidth]{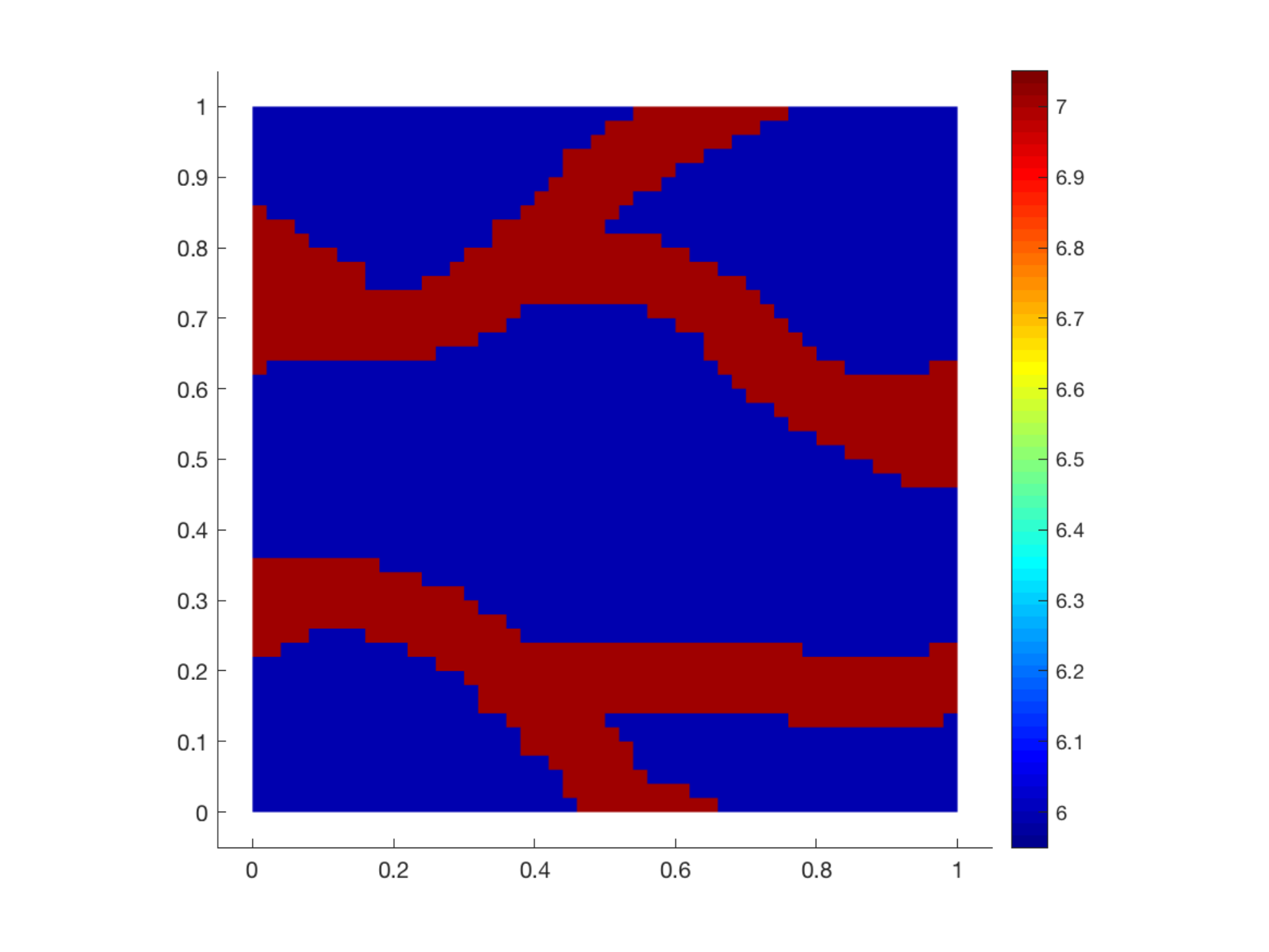} \hfill
		\includegraphics[width=.24\textwidth]{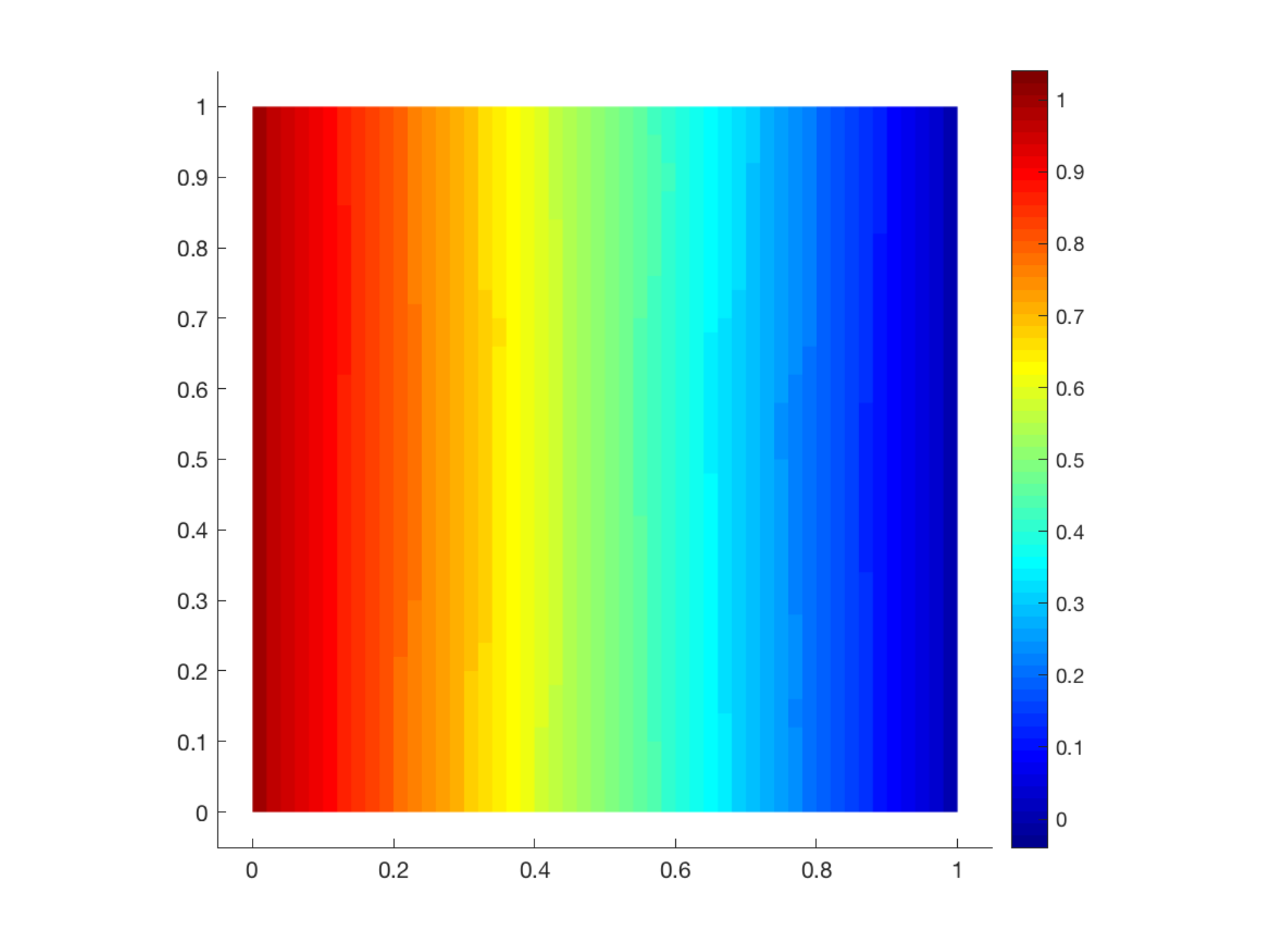}\hfill
		\includegraphics[width=.24\textwidth]{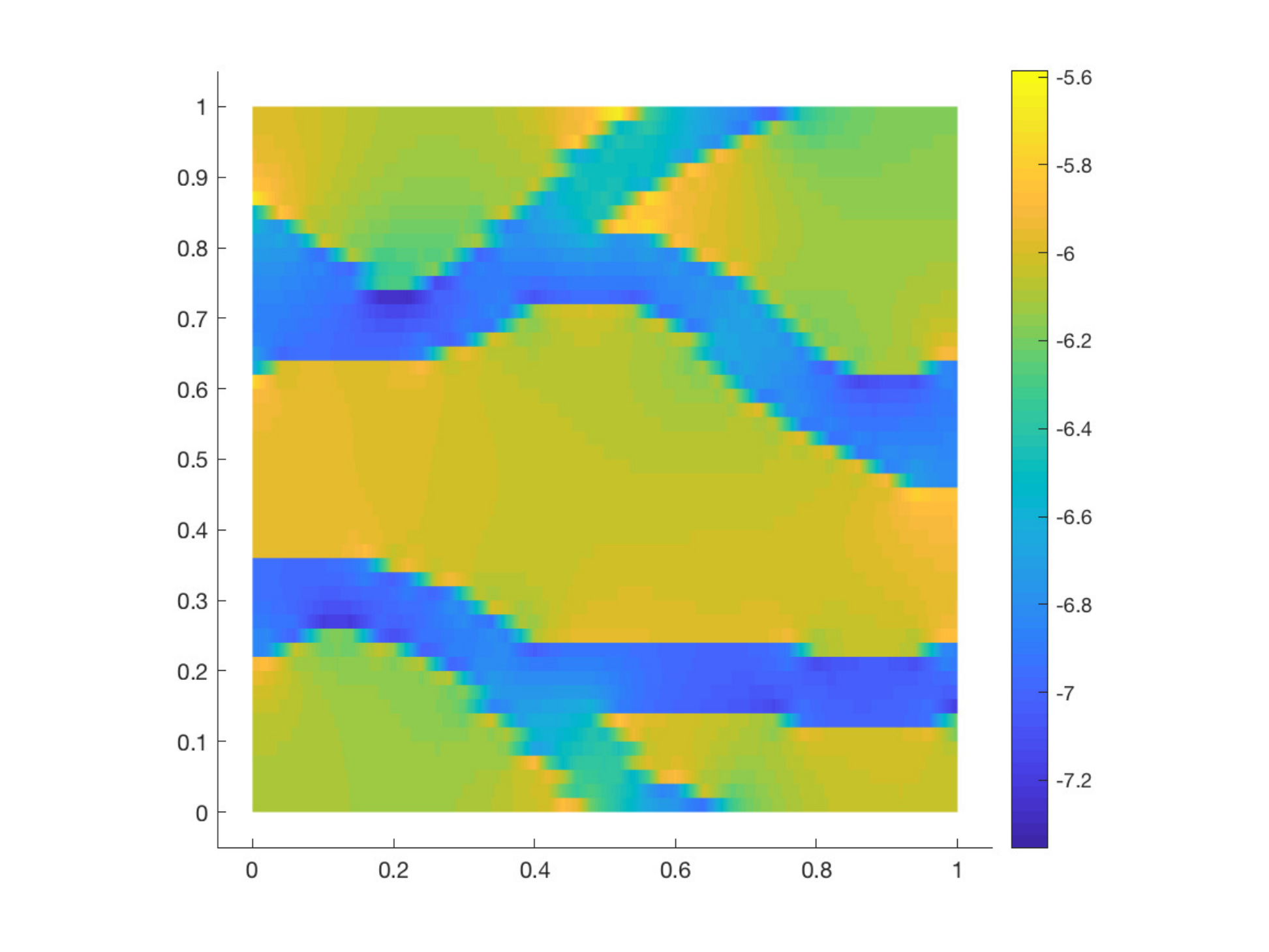}	\hfill	
		\includegraphics[width=.24\textwidth]{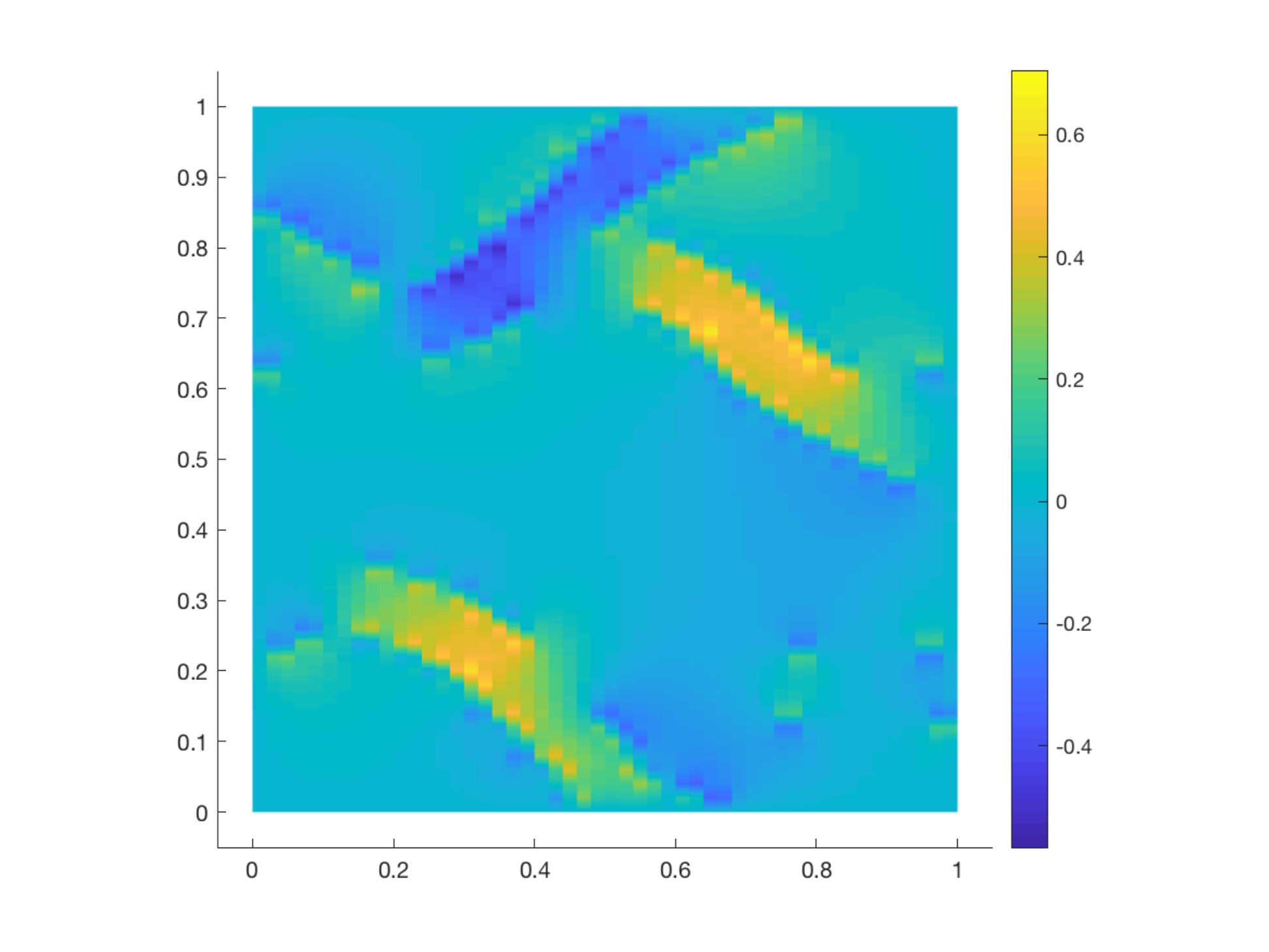}\hfill
		\caption{Illustration example 1}
	\end{subfigure}
			\begin{subfigure}[t]{1.0\textwidth}
		\centering
		\includegraphics[width=.24\textwidth]{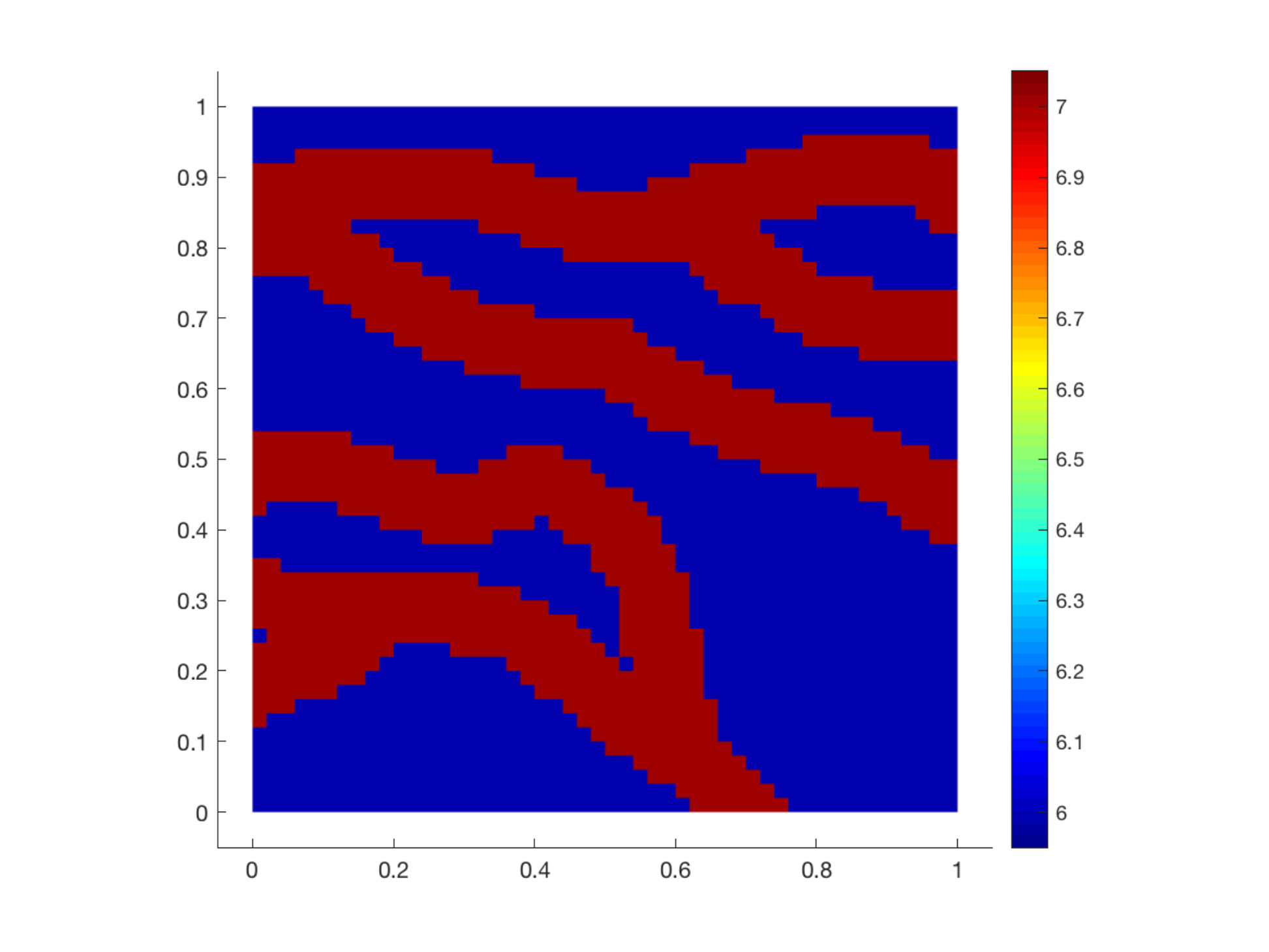} \hfill
		\includegraphics[width=.24\textwidth]{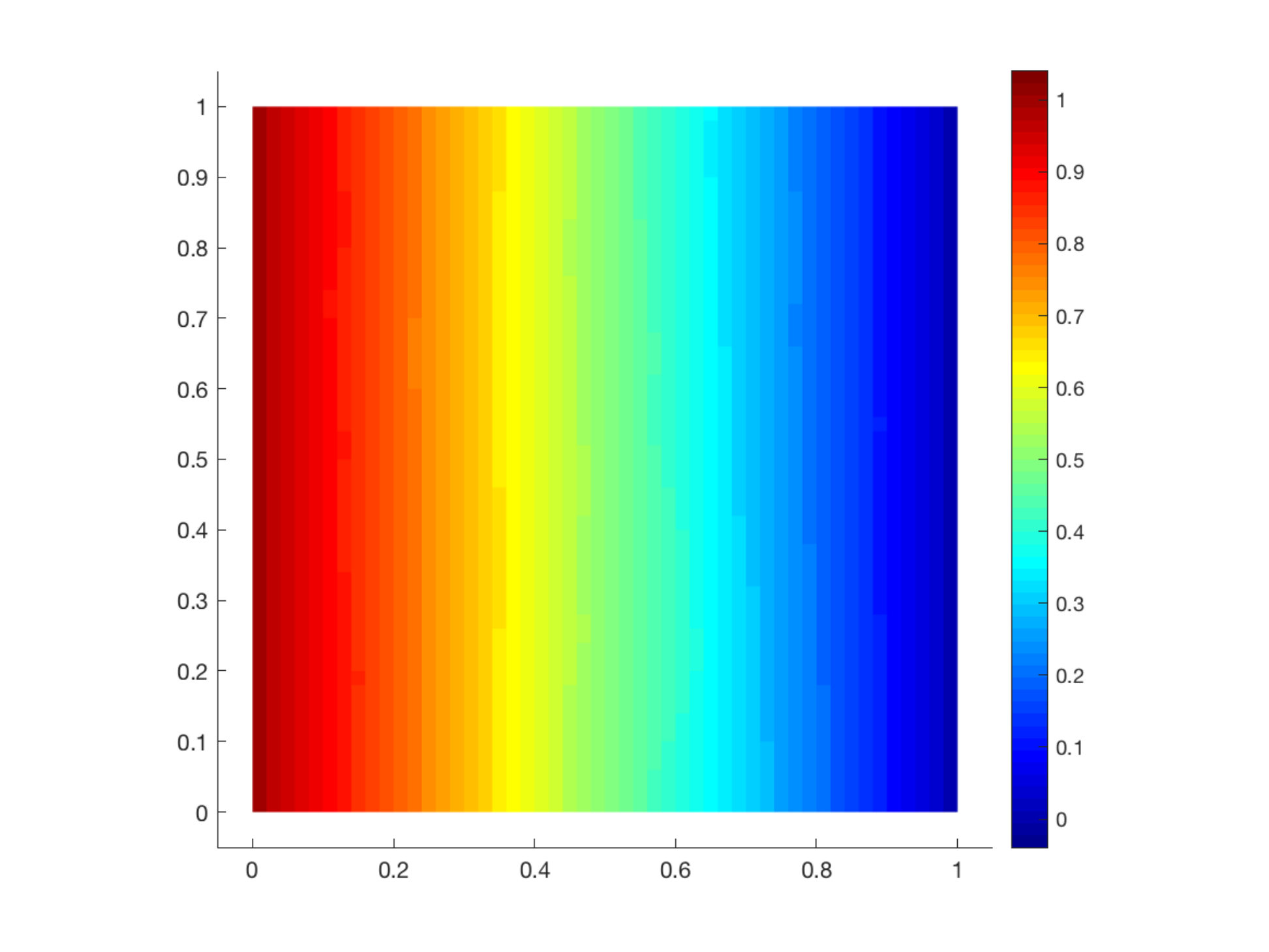}\hfill
		\includegraphics[width=.24\textwidth]{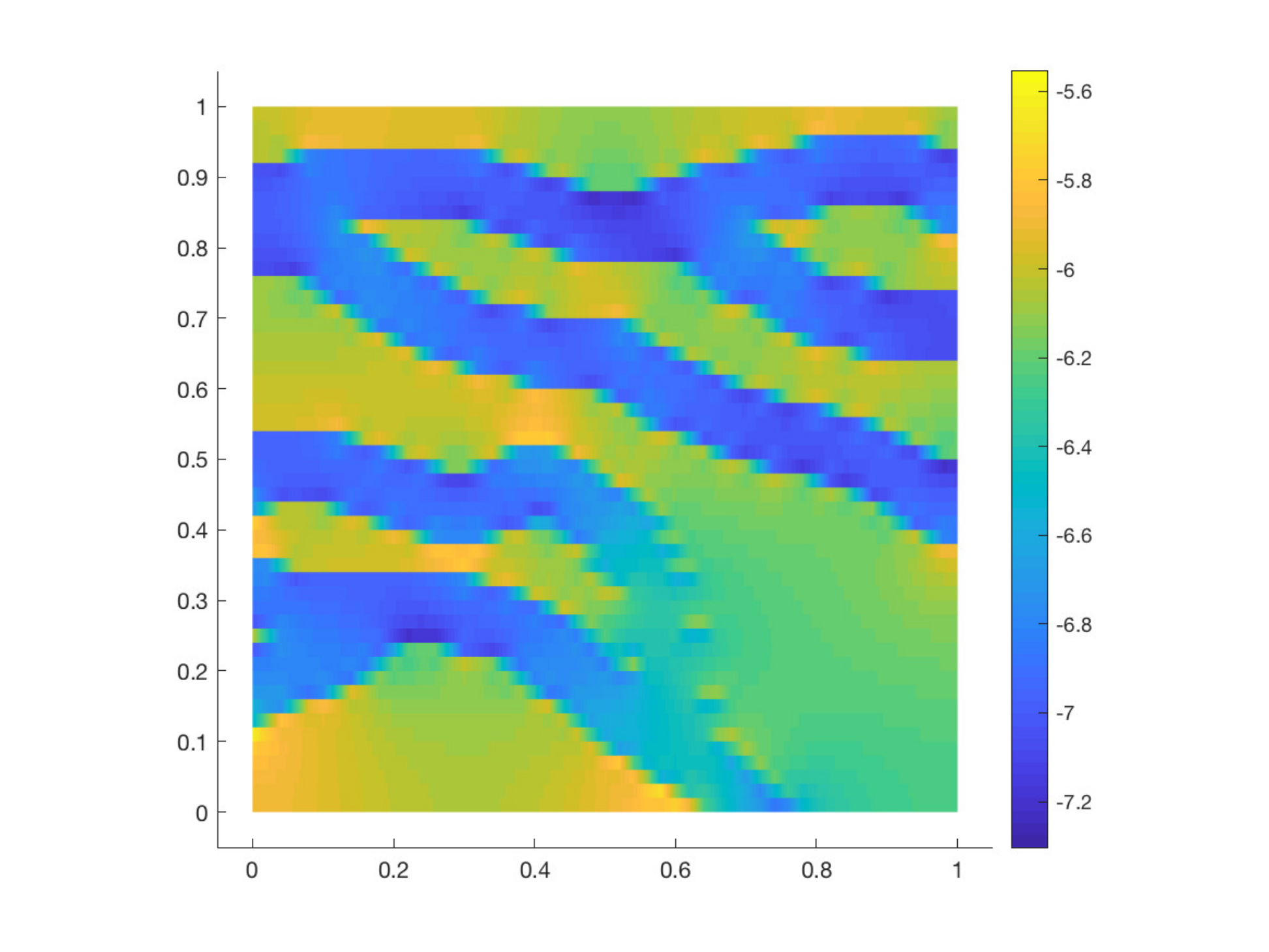}	\hfill	
		\includegraphics[width=.24\textwidth]{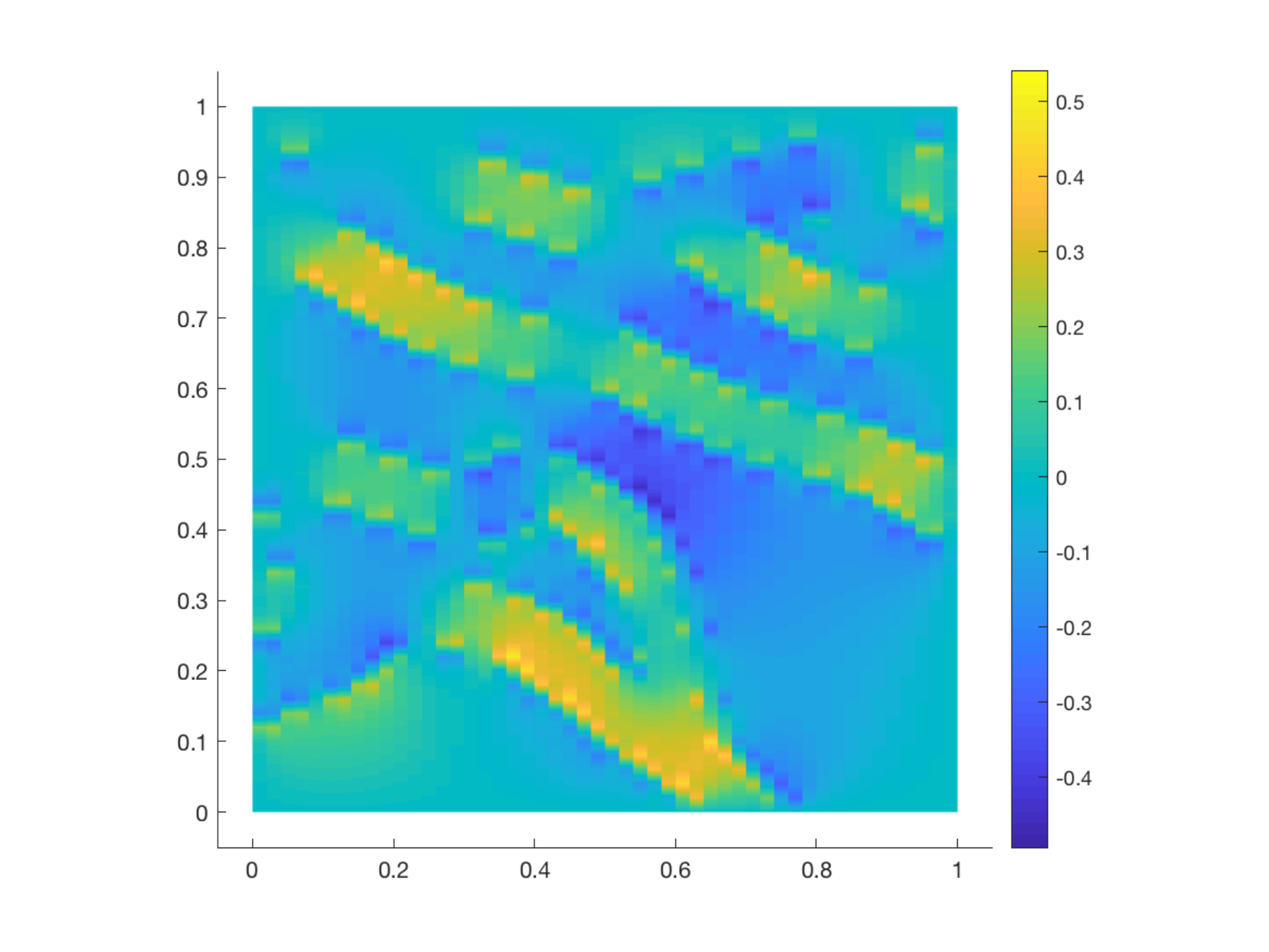}\hfill
		\caption{Illustration example 2}
	\end{subfigure}
	\caption{Illustrations of channelized permeability fields and corresponding solutions. In each subplot, permeability (upper left), pressure solution (upper right), horizontal velocity magnitude (lower left) and vertical velocity magnitude (lower right).}\label{fig:channel}
\end{figure}

We set the sparse rate to be $30\%$, $50\%$, and $70\%$. The performance of our proposed method PSGLD-SA compared with vanilla SGLD, vanilla PSGLD, and SGLD-SA is shown in Table \ref{tab:example3}. We see that PSGLD-SA outperforms SGLD-SA in all three sparse cases. PSGLD-SA also results in more accurate results compared with vanilla PSGLD. The test learning curves are presented in Figure \ref{fig:channel_learning_curves}. It shows that preconditioning helps to improve convergence speed. With stochastic approximation, PSGLD-SA provides better results compared with PSGLD. Some samples are presented in Figure \ref{fig:example3}. It is clear that predicted velocity solutions using PSGLD-SA captures the heterogeneities in the underlying problem, and look very close to the true solutions. However, the predicted solutions obtained from SGLD-SA are not reliable.

\begin{table}[!htb]
	\centering
	\begin{tabular}{|c  |c  | c  | }
		\hline
    &	PSGLD ($e_1$/ $e_2$ \% )& 	SGLD ($e_1$/ $e_2$ \% )  \\  \hline
     Dense &3.31/2.73  &6.39/ 4.94  \\  \hline
	&	PSGLD-SA ($e_1$/ $e_2$ \% )& 	SGLD-SA ($e_1$/ $e_2$ \% )  \\  \hline
    Sparse rate 30\% &2.73/2.13  &3.94/ 3.12  \\  \hline
	Sparse rate 50\% & 2.75/2.16 &3.91 /3.05     \\  \hline
	Sparse rate 70\%  &2.71/1.95  &4.46/3.57      \\  \hline
	\end{tabular}
	\caption{Channelized permeability fields. Mean errors between the true and predicted velocity using proposed SGLD, PSGLD, SGLD-SA, and PSGLD-SA. Mean errors among $600$ testing samples.}\label{tab:example3}
\end{table}

\begin{figure}[!ht]
	\centering
		\includegraphics[scale=0.3]{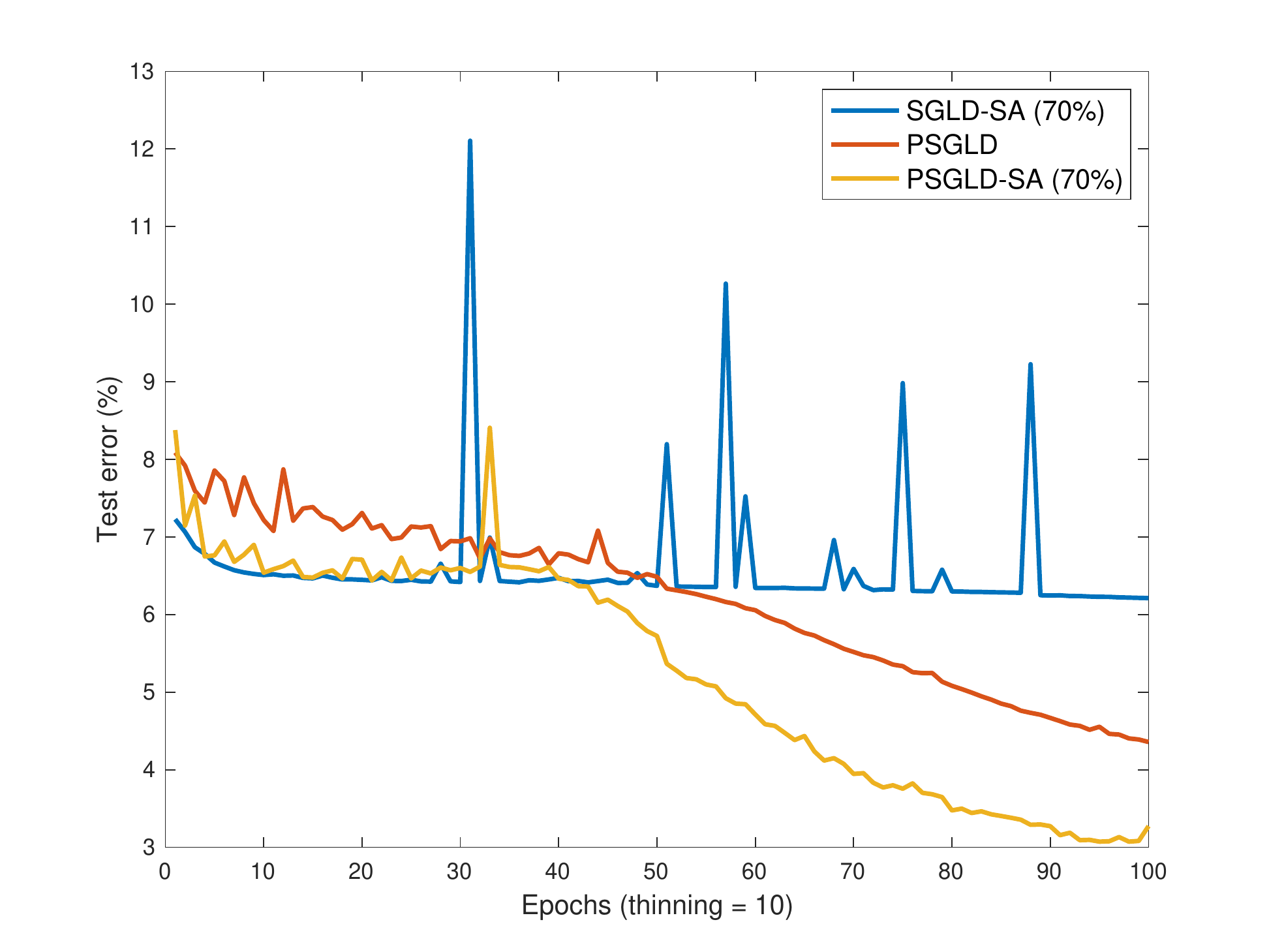}
	\caption{Channelized media, learning curves. Comparison among test errors SGLD with sparse approximation, vanilla PSGLD, and PSGLD with sparse approximation.}\label{fig:channel_learning_curves}
\end{figure}

\begin{figure}[!ht]
	\begin{subfigure}[t]{0.9\textwidth}
		\centering
		\includegraphics[scale=0.7]{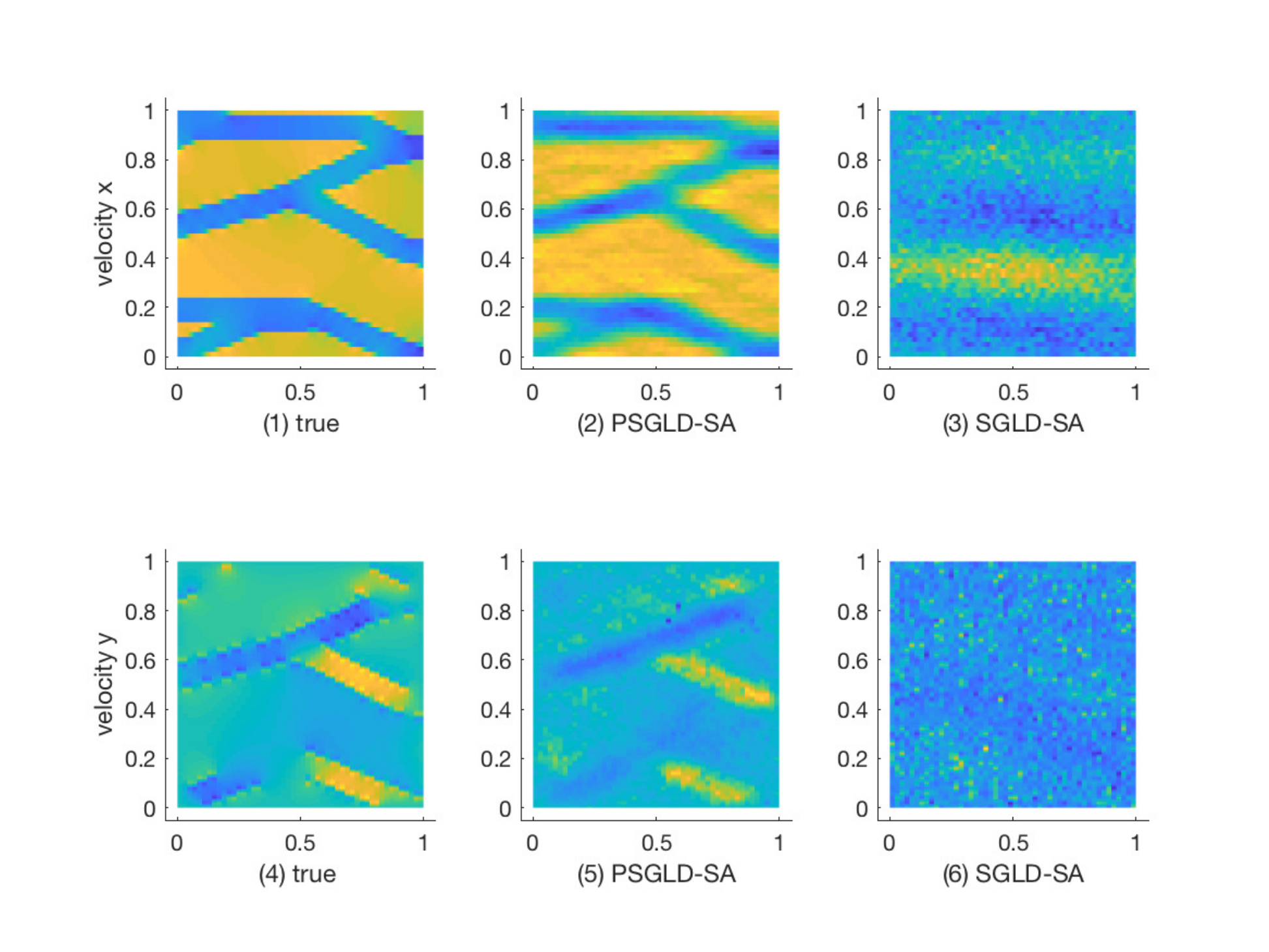}
		\caption{Channelized permeability field, test sample 1,sparse rate 50\%.}
	\end{subfigure}%
	
	\begin{subfigure}[t]{0.9\textwidth}
		\centering
		\includegraphics[scale=0.7]{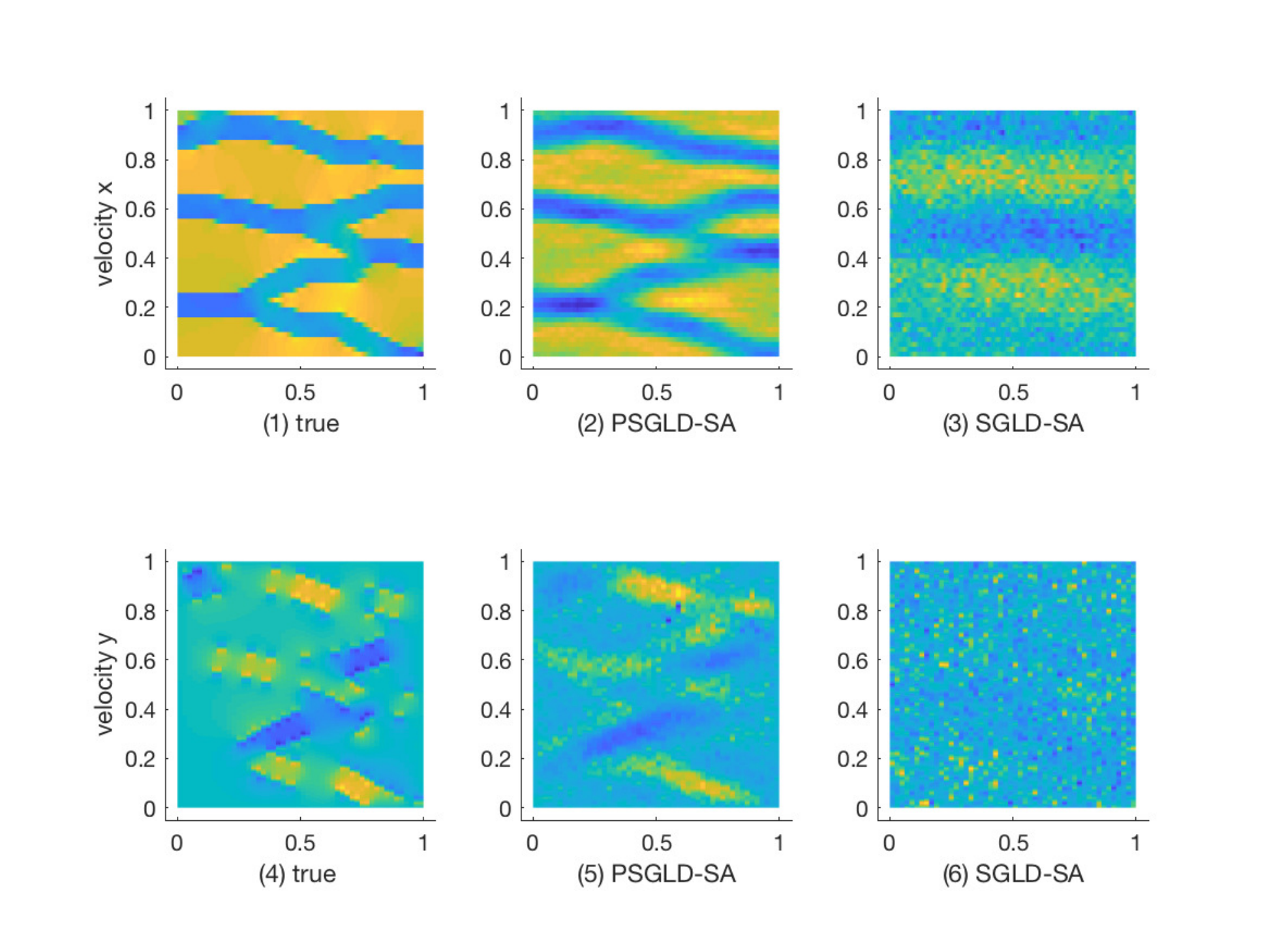}	
		\caption{Channelized permeability field, test sample 2, sparse rate 70\%.}
	\end{subfigure}
	\caption{Channelized permeability field. True and prediction solutions. In each sub-figure, first row represents horizontal velocity solution magnitude, second row represents vertical velocity solution magnitude.} \label{fig:example3}
\end{figure}

\section{Conclusion}\label{sec:conclusion}
We proposed a Bayesian sparse learning algorithm, where the model parameters are adaptively trained from a Bayesian mixture deep neural network, and the latent variables are smoothly learned through optimization. The Bayesian hierarchical model adopts SSGL priors, and samples are generated from the posterior using preconditioned Stochastic gradient descent Markov Chain Monte Carlo (PSGLD). PSGLD incorporates local curvature information in parameter updating scheme, such that constant step size is adequate and slow mixing can be avoided. Due to the diagonal form of the preconditioning matrix, PSGLD needs less computational and storage cost compared to SGRLD. Moreover, we apply stochastic approximation techniques in the sequentially updated preconditioning matrix, the bias on the MSE introduced due to ignoring a correction term will approach zero. The convergence of the proposed algorithm is discussed. Numerical simulations are performed to learn the solutions of elliptic PDE with heterogeneous coefficients. Sparse learning with preconditioned SGLD sampling algorithm is shown to be helpful to accelerate the learning process and the trained sparse models which can be used as computational efficient surrogates for solving the underlying PDE. 
The algorithm can also be extended to solve other heterogeneous problems, and applied to the multi-fidelity framework. Moreover, we may construct appropriate network structure and enforce sparsity according to physical information, such that we can interpret the sparse network obtained physically.

\section*{Acknowledgement}
We gratefully acknowledge the support from the National Science Foundation (DMS-1555072, DMS-1736364, CMMI-1634832, and CMMI-1560834), and Brookhaven National Laboratory Subcontract 382247. The authors would also like to acknowledge the support from NVIDIA Corporation for the donation of the Titan Xp GPU used for this research.

\bibliographystyle{siam} 
\bibliography{references.bib}

\begin{appendices}

\section{Convergence of latent variables}\label{app:theta}

\begin{assumption}
	The step size $\{\omega_k\}$ in the update formula for latent variables satisfies $	\sum_{k=1}^{\infty} \omega_k  = +\infty, \; \sum_{k=1}^{\infty} \omega_k^2  < +\infty $, moreover,
	\begin{equation*}
	{\lim \inf}_{k\rightarrow \infty} 2\delta \frac{\omega_k}{\omega_{k+1}} + \frac{\omega_{k+1}-\omega_k}{\omega_{k+1}^2} > 0 
	\end{equation*}
\end{assumption}
In practice, one can choose $\omega_{k} = c_1(k+c_2)^{-\gamma}$ for $\gamma \in (0,1])$ and constants $c_1, c_2$.

\begin{assumption} \label{assumption:poisson}
	For all $\boldsymbol \theta\in \Theta$, there exists a function $\mu_{\boldsymbol \theta}(\boldsymbol \beta)$ that solves the Poisson equation 
	$\mu_{\boldsymbol \theta}(\boldsymbol \beta) - \Pi_{\boldsymbol \theta} \mu_{\boldsymbol \theta}(\boldsymbol \beta) = H(\boldsymbol \theta,\boldsymbol \beta) - h(\boldsymbol \theta)$. 
	There exists a constant $C$ such that
	\begin{align*}
	\mathbb{E} ||  \Pi_{\boldsymbol \theta} \mu_{\boldsymbol \theta}(\boldsymbol \beta) || &\leq C\\
	\mathbb{E} ||  \Pi_{\boldsymbol \theta} \mu_{\boldsymbol \theta}(\boldsymbol \beta) - \Pi_{{\boldsymbol \theta}'} \mu_{{\boldsymbol \theta}'}(\boldsymbol \beta)||  &\leq C ||\boldsymbol \theta-{\boldsymbol \theta}'||
	\end{align*}
\end{assumption}

\begin{lemma}\label{lemma:sequence}
	There exists $\lambda_0$ and $k_0$ such that $\forall \lambda \geq\lambda_0$ and $\forall k \geq k_0$, the sequence $\displaystyle{ \{\psi_k \}_{k=1}^{\infty} }$  with $\displaystyle{ \psi_k = \lambda \omega_k + 2C_2/\delta \sup_{i\geq k_0}  \triangle_i}$ satisfies
	\begin{equation}\label{eq:psi}
	\psi_{k+1} \geq (1-2\delta\omega_{k+1} + C_1\omega_{k+1}^2 )	\psi_k+ C_1 \omega_{k+1}^2 + 2C_2 \triangle_k  \omega_{k+1}   
	\end{equation}
\end{lemma}

\begin{proof}
	Plug in $\displaystyle{ \psi_k = \lambda \omega_k + 2C_2/\delta \sup_{i\geq k_0}  \triangle_i}$ in equation \eqref{eq:psi}, 
		\begin{equation*}
	(\lambda \omega_{k+1} + 2C_2/\delta \sup_{i\geq k_0}  \triangle_i) \geq (1-2\delta \omega_{k+1} +C_1\omega_{k+1}^2 )	(\lambda \omega_k + 2C_2/\delta \sup_{i\geq k_0}  \triangle_i)+ C_1\omega_{k+1}^2 + 2C_2\triangle_k  \omega_{k+1} 
	\end{equation*}
	Rearranging terms, we need to show
			\begin{equation*}
\lambda	( \omega_{k+1} -\omega_k + 2 \delta \omega_k \omega_{k+1} - C_1\omega_k \omega_{k+1}^2 ) \geq (-2\delta\omega_{k+1} +  C_1\omega_{k+1}^2)	  2C_2/\delta \sup_{i\geq k_0}  \triangle_i + C_1\omega_{k+1}^2 + 2C_2\triangle_k  \omega_{k+1} 
	\end{equation*}
	That is,
\begin{equation}\label{eq:seq2}
	\lambda	( 2\delta \frac{\omega_k}{\omega_{k+1}} + \frac{\omega_{k+1}-\omega_k}{\omega_{k+1}^2}  - C_1\omega_k) \omega_{k+1}^2 \geq   \omega_{k+1}^2(  C_1 + 2C_1C_2/\delta \sup_{i\geq k_0}  \triangle_i) - (\sup_{i\geq k_0}  \triangle_i -\triangle_k   )2 C_2\omega_{k+1}
	\end{equation}
	Let $\displaystyle{ M_1 = 	{\lim \inf}_{k\rightarrow \infty} 2\frac{\omega_k}{\omega_{k+1}} + \frac{\omega_{k+1}-\omega_k}{\omega_{k+1}^2} } $, we see the left hand side of \eqref{eq:seq2} is greater than equal to $\displaystyle{ \lambda	( M_1 - C_1\omega_k) \omega_{k+1}^2 }$. And use the fact that $\sup_{i\geq k_0}  \triangle_i - \triangle_k \geq 0$, we have the right hand side of \eqref{eq:seq2} is less than equal to $\omega_{k+1}^2(  C_1 + 2C_1C_2/\delta \sup_{i\geq k_0}  \triangle_i)$. Now it is suffices to show that 
	\begin{equation}\label{eq:seq3}
	\lambda	( M_1 - C_1\omega_k) \omega_{k+1}^2 \geq   \omega_{k+1}^2(  C_1 + 2C_1C_2 \sup_{i\geq k_0}  \triangle_i)
	\end{equation}
	By choosing $\lambda_0$ and $k_0$ such that $\displaystyle{\omega_{k_0} \leq \frac{M_1}{2C_1}}$, and $\displaystyle{\lambda_0 = \frac{4C_1C_2 \sup_{i\geq k_0}  \triangle_i +2 C_1}{M_1}}$, \eqref{eq:seq2} holds, thus the desired inequality \eqref{eq:psi} holds.

\end{proof}

\begin{theorem} \label{thm:main_G}
	
	Suppose Assumptions 1-2 hold, with assumptions and Lemmas 1-2, Propositions 1-3 in \cite{sgld-sa}, we have the sequence $\{{\boldsymbol \theta}_k\}_{k=1}^{\infty}$ converge to ${\boldsymbol \theta}_*$, and there exist a sufficiently large $k_0$ such that
	\begin{equation*}
	\mathbb{E}||\boldsymbol \theta_k - \boldsymbol \theta_*||^2 = \mathcal{O} (\lambda \omega_k + \sup_{i\geq k_0}	\mathbb{E} ||\Delta(n, \boldsymbol \theta_i,\boldsymbol \beta_{i+1})||) 
	\end{equation*}
	
\end{theorem}

	\begin{proof}
		Denote by $E_k = \boldsymbol \theta_k-\boldsymbol \theta_*$, we have
		\begin{equation}\label{eq:thm-error}
		||E_{k+1}||^2 = ||E_k||^2 + \omega_{k+1}^2|| \tilde{H}(\boldsymbol{\beta_{k+1}}, \boldsymbol \theta_k) ||^2  + 2\omega_{k+1} \mathbb{E} \langle E_k, \tilde{H}(\boldsymbol{\beta_{k+1}}, \boldsymbol \theta_k)  \rangle
		\end{equation}
		For the third term in \eqref{eq:thm-error}, we have
		\begin{align*} 
		&\langle E_k, \tilde{H}(\boldsymbol{\beta_{k+1}}, \boldsymbol \theta_k)  \rangle = \langle E_k, H(\boldsymbol{\beta_{k+1}}, \boldsymbol \theta_k)  + \Delta(n, \boldsymbol \theta_i,\boldsymbol \beta_{i+1})  \rangle\\
		\leq & -||E_k||^2 +   \langle E_k, \mu_{\boldsymbol \theta_k}(\boldsymbol \beta_{k+1}) - \Pi_{\boldsymbol \theta_k} \mu_{\boldsymbol \theta_k}(\boldsymbol \beta_{k}) \rangle  +  \langle E_k,\Pi_{\boldsymbol \theta_k} \mu_{\boldsymbol \theta_k}(\boldsymbol \beta_{k})- \Pi_{\boldsymbol \theta_{k-1}} \mu_{\boldsymbol \theta_{k-1}}(\boldsymbol \beta_k) \rangle  \\
		&  +  \langle E_k, \Pi_{\boldsymbol \theta_{k-1}} \mu_{\boldsymbol \theta_{k-1}}(\boldsymbol \beta_k)- \Pi_{\boldsymbol \theta_k} \mu_{\boldsymbol \theta_k}(\boldsymbol \beta_{k+1}) \rangle  +   ||E_k|| \triangle_k 
		\end{align*}
		where $||\Delta(n, \boldsymbol \theta_k,\boldsymbol \beta_{k+1})|| =  \triangle_k$.
		 Following a similar proof as in \cite{sgld-sa}, we have
		 \begin{equation*}
		    2\omega_{k+1} \mathbb{E} \langle E_k, \tilde{H}(\boldsymbol{\beta_{k+1}}, \boldsymbol \theta_k)  \rangle C_2 \omega_{k+1}
		 \end{equation*}
		 Thus,
		 	\begin{equation*}
			\mathbb{E} ||E_{k+1}||^2 \leq (1-2\delta \omega_{k+1} + C_1\omega_{k+1}^2 )	\mathbb{E} ||E_k||^2 +  C_1\omega_{k+1}^2 + 2 C_2 \triangle_k \omega_{k+1}   + 2\omega_{k+1 }	\mathbb{E} [z_k - z_{k+1}]\rangle
		 \end{equation*}
		 where we use the fact that $|| \tilde{H}(\boldsymbol{\beta_{k+1}}, \boldsymbol \theta_k) ||^2 \leq C_1 (1+||E_k||^2 )$.
		 
		 According to Lemma \ref{lemma:sequence}, there exists $\lambda_0$, $k_0$ such that
		 \begin{equation*}
		 \mathbb{E} ||E_{k_0}||^2 \leq \psi_{k_0} = \lambda_0 \omega_{k_0} + 2C_2/\delta \sup_{i\geq k_0}  \triangle_i
		 \end{equation*}
		 Thus,
		 \begin{equation} \label{eq:thm_Ekbd}
		 \mathbb{E} ||E_{k}||^2 \leq \psi_{k} +  \mathbb{E}[\sum_{j=k_0+1}^k \Lambda_j^k (z_{j+1} - z_j)] 
		 \end{equation}
		 From Assumption \ref{assumption:poisson} and that $\boldsymbol \theta$ is uniformly bounded, there exists $C_3>0$
		  \begin{equation*}
		 \mathbb{E}[|z_k|]  = \mathbb{E}\left[\left|  \langle E_k, \Pi_{\boldsymbol \theta_{k-1}} \mu_{\boldsymbol \theta_{k-1}}(\boldsymbol \beta_{k}) \rangle  \right| \right] \leq  \mathbb{E} ||E_k||  \mathbb{E}\left[\left|  \Pi_{\boldsymbol \theta_{k-1}} \mu_{\boldsymbol \theta_{k-1}}(\boldsymbol \beta_{k})   \right| \right] \leq C_3
		 \end{equation*}
		 
		 Moreover, due to the fact that $k_0$ is an integer satisfying 
	\begin{equation*}
	\inf_{k\geq k_0} \frac{\omega_{k+1} - \omega_k}{\omega_k\omega_{k+1} } + 2\delta -C_1\omega_{k+1}  > 0.
	\end{equation*}
	Then $\forall k \geq k_0$, the sequence $\{ \Lambda_k^K\}_{k=k_0}^{K}$ is increasing, where
	\begin{equation*}
	 \Lambda_k^K = 	\begin{cases}
	 \displaystyle{ 2\omega_k \prod_{j=k}^{K-1} (1-2\omega_{k+1}\delta + C_1\omega_{k+1}^2)}, & \text{if } \displaystyle{  k < K} \\
	 2\omega_k, & \text{if } \displaystyle{  k = K} 
	 \end{cases} 
	\end{equation*}
		 Thus, 
		 \begin{align*}
		   \mathbb{E}\left[ \left| \sum_{j=k_0+1}^k \Lambda_j^k (z_{j+1} - z_j)  \right| \right] &=  \mathbb{E}\left[ \left| \sum_{j=k_0+1}^{k-1}  ( \Lambda_{j+1}^k - \Lambda_j^k) z_j +\Lambda_{k_0+1}^k  z_{k_0} -\Lambda_k^k  z_k \right| \right]  \\
		   &\leq  ( \Lambda_{k}^k - \Lambda_{k_0+1}^k) C_3 + \Lambda_{k_0+1}^k C_3 + \Lambda_k^k C_3 \\
		  & = 2\Lambda_k^k C_3  \leq 4 C_3 \omega_k
		 \end{align*}
		 
Then the inequality \eqref{eq:thm_Ekbd} can be further bounded as
		 \begin{align*}
\mathbb{E} ||E_{k}||^2& \leq \lambda_0 \omega_k + C_2 \sup_{i \geq k_0}  \triangle_i + 4C_3  \omega_{k} \\
&= \lambda \omega_k + C_2 \sup_{i \geq k_0}  \triangle_i
\end{align*}
where $\lambda = \lambda_0 + 4C_3$.
	\end{proof}

\end{appendices}

\end{document}